\documentclass[11pt]{amsart}

\usepackage{amsmath, amssymb, amsthm}
\usepackage[T1]{fontenc}
\usepackage[utf8]{inputenc}
\usepackage[all]{xy}
\usepackage[
  margin=1.5cm,
  includefoot,
  footskip=30pt,
  ]{geometry}
\usepackage[colorlinks=true,citecolor=cyan,linkcolor=magenta, urlcolor=blue, filecolor=green, backref=page]{hyperref}
\usepackage{hyperref}
\usepackage{indentfirst}

\newcommand{\cD}{\mathcal{D}}
\newcommand{\cE}{\mathcal{E}}

\newcommand{\cG}{\mathcal{G}}

\newcommand{\cJ}{\mathcal{J}}

\newcommand{\cLC}{\mathcal{LC}}
\newcommand{\cM}{\mathcal{M}}

\newcommand{\cS}{\mathcal{S}}

\newcommand{\bF}{\mathbb{F}}

\newcommand{\bQ}{\mathbb{Q}}
\newcommand{\bR}{\mathbb{R}}

\newcommand{\bZ}{\mathbb{Z}}

\renewcommand{\Re}{\operatorname{Re}}

\DeclareMathOperator{\Aut}{Aut}
\DeclareMathOperator{\Area}{Area}

\DeclareMathOperator{\Tr}{Tr}

\DeclareMathOperator{\dual}{dual}
\DeclareMathOperator{\hyp}{hyp}

\DeclareMathOperator{\trace}{trace}

\theoremstyle{plain}
\newtheorem{theorem}{Theorem}[section]
\newtheorem{lemma}[theorem]{Lemma}
\newtheorem{proposition}[theorem]{Proposition}
\newtheorem{corollary}[theorem]{Corollary}

\theoremstyle{definition} 
\newtheorem*{definition}{Definition}
\newtheorem{theoremmain}{Theorem}
\newtheorem{corollarymain}[theoremmain]{Corollary}
\newtheorem{conjecture}{Conjecture}

\theoremstyle{remark} 
\newtheorem{remark}{Remark}
\newtheorem{example}[remark]{Example}

\DeclareFontFamily{U}{wncy}{}
\DeclareFontShape{U}{wncy}{m}{n}{<->wncyr10}{}
\DeclareSymbolFont{mcy}{U}{wncy}{m}{n}
\DeclareMathSymbol{\Sha}{\mathord}{mcy}{"58}

\newcommand{\lp}{\left(}
\newcommand{\rp}{\right)}

\newcommand{\lbrb}[1]{\lp #1 \rp}
\newcommand{\lcrc}[1]{\left\{ #1 \right\}}

\title{Average analytic rank of elliptic curves with prescribed torsion}
\author[Peter Jaehyun Cho]{Peter J. Cho}
\address{Department of Mathematical Sciences, Ulsan National Institute of Science and Technology, UNIST-gil 50, Ulsan 44919, Korea}
\email{petercho@unist.ac.kr}

\author[Keunyoung Jeong]{Keunyoung Jeong}
\address{Department of Mathematical Sciences, Ulsan National Institute of Science and Technology, UNIST-gil 50, Ulsan 44919, Korea}
\email{kyjeong@unist.ac.kr}
\thanks{Peter J. Cho and Keunyoung Jeong acknowledge the support by the NRF grant funded by the  Korea government(MSIT)  (No. 2019R1F1A1062599, 2019R1C1C1004264) respectively and they are also supported by the Basic  Science Research Program(2020R1A4A1016649) together.}

\begin{document}

\maketitle

\begin{abstract}
We show that average analytic rank of elliptic curves with prescribed torsion $G$ is bounded for every torsion group $G$ under GRH for elliptic curve $L$-functions.  
\end{abstract}

\section{Introduction}
 
The distribution of (algebraic or analytic) ranks of elliptic curves defined over $\mathbb{Q}$ is one of the most interesting problems in number theory.  
One of important features of the distribution is the average rank of elliptic curves. Let us start with our model for elliptic curves. Our elliptic curves defined over $\bQ$ are represented by for a pair $(A,B)$ of integers with $4A^3+27B^2 \neq 0$
\begin{align*}
E_{A,B}: y^2=x^2+Ax+B
\end{align*} 
such that there is no prime $p$ with $p^4\mid A$ and $p^6 \mid B$. Let $\cE$ be the set of all such pairs and $\cE$ has a bijection with the set of $\mathbb{Q}$-isomorphism classes of elliptic curves over $\mathbb{Q}$. Then, we can order elliptic curves by the naive height:
\begin{align*}
\cE(X)=\left\{E_{A,B} \in \cE :  |A|\leq X^\frac 13,\, |B|\leq X^\frac 12 \right \}.
\end{align*}
We can define the average rank of elliptic curves as the limit of the average rank over $\cE(X)$ as $X$ goes to infinity if it exists.
It is widely believed that the following conjecture initially proposed by Goldfeld \cite{Goldfeld} would be true.
\begin{conjecture}[Minimalist conjecture] \label{folklore}
The proportion of elliptic curves with rank $0$ and the proportion of elliptic curves with rank $1$ are both $\frac 12$.
\end{conjecture}
Recently, Park, Poonen, Voight, and Wood \cite{PPVW} has brought out a more refined conjecture\footnote{We note that Conjecture \ref{PPVWconj} is also suggested by \cite{Wat15, WatDis} with a different heuristic method.
} which
not only claims Conjecture \ref{folklore} but also proposes the number of elliptic curves with algebraic rank $\geq r$ for $1 \leq r \leq 20$.

\begin{conjecture} \cite[Corollary 7.2.6, Theorem 7.3.3]{PPVW} \label{PPVWconj}
\begin{enumerate}
\item The proportion of elliptic curves with algebraic rank $0$ and the proportion of elliptic curves with algebraic rank $1$ are both $\frac 12$.
\item There are only finitely many elliptic curves with algebraic rank $ >21$.
\item For $1 \leq r \leq 20$, the proportion of elliptic curves over $\bQ$ with algebraic rank $\geq r$ and height $\leq X$
is $X^{\frac{21-r}{24} + o(1)}$.
\end{enumerate}
\end{conjecture}

A major breakthrough for Conjecture \ref{folklore} was made by Bhargava and Shankar \cite{BS15,BS}. They showed that the proportion of elliptic curves with
algebraic rank $\leq 1$ is at least $0.8375$ and with algebraic rank $0$ is at least $0.2062$. 
For the average analytic rank, Brumer \cite{Bru} showed that it is bounded by 2.3 under GRH for elliptic curve $L$-functions. This bound was lowered to $2$ and $\frac{25}{14}$ by Heath-Brown \cite{Hea} and Young \cite{Y} respectively.

On the other hand, Harron and Snowden \cite{HS} counted elliptic curves with prescribed torsion $G$. 
From now on, we say that an elliptic curve $E$ over $\bQ$ has torsion $G$ if $E(\bQ)$ contains a subgroup isomorphic to $G$.

By a work of Mazur, $G$ is one of the groups
\begin{equation*}
	\bZ/n\bZ, \qquad \bZ/2\bZ \times \bZ/2m\bZ
\end{equation*}
for $n \in \lcrc{1, 2,  \cdots 10, 12}$ and $m \in \lcrc{1, 2, 3, 4}$.
Let 
\begin{equation*}
\mathcal{G}_{\leq 4} := \lcrc{\bZ/2\bZ, \bZ/3\bZ, \bZ/4\bZ, \bZ/2\bZ \times \bZ/2\bZ}
\end{equation*}
and $\mathcal{G}_{\geq 5}$ be the set of torsion groups of order $\geq 5$.
We remark that elliptic curves with torsion $G \in \cG_{\geq 5}$ can be parametrized by the Tate's normal form (See \S 2). 
We often use $n$ and $2 \times 2m$ in place of $G = \bZ/n\bZ$ and $\bZ/2\bZ \times \bZ/2m\bZ$ to ease the notation.

Let
\begin{align*}
\cE_G(X)=\left\{E_{A,B} \in \cE(X):  E(\mathbb{Q}) \geq G \right \}.
\end{align*}
Harron and Snowden showed that 
$$
\lim_{X \rightarrow \infty } \frac{\log\left| \cE_G(X) \right|}{\log X} =\frac{1}{d(G)},
$$
where $d(G)$ is given in Table \ref{tab:}. Furthermore, for $G=\bZ/2\bZ$ and $\bZ/3\bZ$, they obtained the cardinality of $\cE_G(X)$ with a power-saving error term. 
\begin{table}[] 
\caption{}
\begin{tabular}{rrrrrrrr} 
\hline
                 $G$  &                $d(G)$  &   & $G$ &            $d(G)$  & & $G$ & $d(G)$\\ \hline
                  $0$                &  $ 6/5$ & \phantom{ddd}    &$\bZ/6\bZ$  & $6$   & \phantom{ddd}   &   $\bZ/12\bZ$                           & $24$ \\
                  $\bZ/2\bZ$&  $2$      &     &$\bZ/7\bZ$&  $12$     &   &$\bZ/2\bZ \times \bZ/2\bZ$ &  $3$\\
                  $\bZ/3\bZ$&  $3$      &   & $\bZ/8\bZ$&  $12$     &  & $\bZ/2\bZ \times \bZ/4\bZ$ & $6$\\
                  $\bZ/4\bZ$&  $4$      &   &$\bZ/9\bZ$&  $18$     &   &$\bZ/2\bZ \times \bZ/6\bZ$ & $12$\\
                  $\bZ/5\bZ$&  $6$      &   & $\bZ/10\bZ$&  $18$    &  & $\bZ/2\bZ \times \bZ/8\bZ$ & $24$ \\ \hline
\end{tabular} 
\label{tab:}
\end{table}

Not much is known about the distribution of (algebraic or analytic) ranks of elliptic curves with prescribed torsion group $G$. In \cite[\S 8.3]{PPVW}, they give an upper bound of algebraic ranks of elliptic curves in $\cE_G$ but do not give a statement on the distribution of ranks in $\cE_G$ other than this. 
In their preprint, Bhargava and Ho \cite[Theorem 1.1]{BH} obtained bounds for the average algebraic rank of the families of elliptic curves with marked torsion point $(0,0)$ of order 2 and 3 respectively, which are $7/6$ and $3/2$.  
We show that for any torsion group $G$ average analytic rank over the family $\cE_G$ is bounded. 

\begin{theoremmain} \label{thm1}
Let $G$ be a torsion group. For $G=\bZ/n\bZ$, $n=7,9,8,9,10,$ and $12$ and $G=\bZ/2\bZ \times \bZ/2m\bZ$, $m=3,4$, we assume the moment conditions $(\ref{moment_torsion 2})$, $(\ref{moment_torsion 3})$. Under GRH for elliptic curve $L$-functions, the average analytic rank over $\cE_G$ is bounded. 
In particular when $|G| \geq 5$, we have a bound $\frac{1}{2} + 5d(G)$.
\end{theoremmain} 


For $G=\bZ/2\bZ$ and $\bZ/2\bZ \times \bZ/2\bZ$, we have additional information on the distribution of analytic ranks. 
First, we can show that there are not many elliptic curves with torsion $G$ with a high rank.
Let $P_G(r_E \geq a)$ denote the probability of elliptic curves with torsion $G$ such that analytic rank $r_E \geq a$. 
\begin{theoremmain}[Theorem \ref{rank-dist}] \label{thm2}
Assume GRH for elliptic curve $L$-functions.  Let $C$ be a positive constant, let $n$ a positive integer. We have, for $G=\bZ/2\bZ$ and $G=\bZ/2\bZ \times \bZ/2\bZ$,
\begin{align*}
P_G\left(r_E \geq  \frac{(1+C)}{\sigma_{2n}} \right)  \leq \frac{\sum_{k=0}^{n}{ {2n} \choose {2k}}\left( \frac 12\right)^{2n-2k}(2k)!\left( \frac 16 \right)^k }{\left( \frac{C}{\sigma_{2n}} \right)^{2n}}.
\end{align*}
where $\sigma_{2n}=\frac{1}{18n}$ and $\frac{1}{20n}$ for $G=\bZ/2\bZ$ and $G=\bZ/2\bZ \times\bZ/2\bZ$ respectively.
In particular, the probabilities $P_{\bZ/2\bZ}(r_E \geq 23)$ and $P_{\bZ/2\bZ \times \bZ/2\bZ }(r_E \geq 25)$ are both at most $0.0234$.
\end{theoremmain}

We note that there is an analogue \cite[Theorem 2]{Hea} of Theorem \ref{thm2} without torsion restriction, which says
\begin{align*} \label{HB bound}
P(r_E \geq a) \ll \left( \frac{5a}{2} \right)^{-\frac{a}{20}}.
\end{align*}
%
%

We can also give an explicit bound on the $n$-th moment of analytic ranks of elliptic curves with torsion $G$.
\begin{theoremmain}[Theorem \ref{thm : moment}]
Assume GRH for elliptic curve $L$-functions. Let $\sigma_n=\frac{1}{9n}$ and $\frac{1}{10n}$ for $G=\bZ/2\bZ$ and $G=\bZ/2\bZ \times\bZ/2\bZ$ respectively. For every positive integer $n$, we have
\begin{align*}
\limsup_{X \rightarrow \infty} \frac{1}{|\cE_{
G}(X)|}\sum_{E \in \cE_G(X)}r_E^n \leq \sum_{S}\left(  9n \right)^{|S^c|}\sum_{\substack{S_2 \subset S  \\ |S_2| \text{even}}} \left( \frac 12  \right)^{|S_2^c|} \left| S_2 \right|! \left(\frac 16\right)^{|S_2|/2}, 
\end{align*}
where $S$ runs over subsets of $\{1,2,3,\dots,n\}$, and $S_2$ runs over subsets of even cardinality of the set $S$. In particular, the average analytic rank of $\cE_{\bZ/2\bZ}$  and that of $\cE_{\bZ/2\bZ \times \bZ/2\bZ}$ are at most $9.5$ and $10.5$ respectively.
\end{theoremmain}

Our rank results are obtained from computation of one-level (or $n$-level) density for the family of elliptic $L$-functions arising from $\cE_G$. Katz and Sarnak's philosophy claims that the one-level density holds for a test function with any compact support and this philosophy combined with our results implies that average analytic rank over $\cE_G$ for any $G$ is bounded by $\frac 12$. 
Since it is widely believed that the root numbers are evenly distributed in $\cE_G$, our one-level density results give small evidence toward the following folklore conjecture.  

\begin{conjecture} 
Let $G$ be a torsion group. The proportion of elliptic curves with rank $0$ in $\cE_G$ and the proportion of elliptic curves with rank $1$ in $\cE_G$ are both $\frac 12$.
\end{conjecture} 
For some numerical data for $G=\bZ/2\bZ \times \bZ/8\bZ$, we refer a result of Chan, Hanselman and Li  \cite{CHL}. Young \cite[\S 8]{Y} also computed bounds for average analytic rank for families of elliptic curves with some prescribed torsion $G$ under not only GRH for elliptic curve $L$-functions but also GRH for Dirichlet $L$-functions and some other assumptions. 

Our approach gives a systematic frame to compute the one-level density for any $G$ using a version of Eichler--Selberg trace formula by Kaplan and Petrow \cite{KP2}. 
This version of Eichler--Selberg trace formula is indispensable to deal with every torsion group $G$. However, to bound up the average rank, we need to count elliptic curves satisfying a local condition. 
A local condition at prime $p$ is a property of an elliptic curve $E$ when reduced modulo $p$. For example, we say that an elliptic curve $E$ satisfies a local condition good, mult, addi or $a$ at a prime $p$ if its reduction modulo $p$ has good reduction, multiplicative reduction, additive reduction or  $a_E(p)=p+1- |E(\mathbb{F}_p) | =a$ respectively.
For torsion groups $G\in \mathcal{G}_{\leq 4}$, we have

\begin{theoremmain} [{Theorem \ref{2tor bad}}] \label{local-small}
For a prime $p \geq 5$, a local condition $\cLC$, and a group $G $ in $ \mathcal{G}_{\leq 4}$, 
\begin{equation*}
|\cE_{G, p}^{\cLC}(X)| = c(G) \cdot  c_{G,\cLC}(p) \cdot \frac{p^{\frac{12}{d(G)}}}{p^{\frac{12}{d(G)}}-1} X^{\frac{1}{d(G)}} + O(h_{G,\cLC}(p, X))
\end{equation*}
where $c_{G,\cLC}(p)$ is a constant depending on $G, p, \cLC$ and $h_{G, \cLC}(p, X)$ is a function whose order of magnitude is less than $pX^{\frac{1}{e(G)}} + p^2 X^{\frac{1}{12}}$.
\end{theoremmain}

For torsion groups $G\in \mathcal{G}_{\geq 5}$, we obtain Theorem \ref{local_count}, an analogue of Theorem \ref{local-small}, based on the work of \cite{CKV} which computes the cardinality of $\cE_{G}(X)$.
As a result of Theorems \ref{local-small} and \ref{local_count}, there are many interesting phenomena. 
One of our motivations in this article was comparing the probability for a local condition under no prescribed torsion with that for the local condition under prescribed torsion. 

\begin{corollarymain}
For $p \geq 5$,  $\cLC \in \{ \text{good, mult, $a$}\}$ and a torsion group $G$, 
we have
\begin{align*} 
\lim_{X \to \infty} \frac{|\cE_{p}^{\cLC}(X)|}{|\cE(X)|} \neq
\lim_{X \to \infty} \frac{|\cE_{G, p}^{\cLC}(X)|}{|\cE_G(X)|}.
\end{align*}
In other words, the three local conditions above and torsion $G$ are not independent.
\end{corollarymain}
The constant $c_{G, \cLC}(p)$ is essentially the probability for an elliptic curve with torsion $G$ to satisfy $\cLC$ at $p$. When $\cLC = \mathrm{mult}$, we can give an interesting interpretation of $c_{G, \cLC}(p)$.

\begin{corollarymain}[{Corollary \ref{p for mult}}]
Let $p$ be a prime $\geq 5$ and $G \in \cG_{\leq 4}$. Then, $c_{G, \mathrm{mult}}(p)$  is proportional to the ratio of the number of the cusps of corresponding modular curve $X_1(N)$ and $X(2)$.
For $G \in \cG_{\geq 5}$, there is a set of primes $p$ of positive density such that 
$c_{G, \mathrm{mult}}(p)$ is proportional to the number of cusps of corresponding modular curves. 
\end{corollarymain}
We note that $c_{G, \mathrm{mult}}(p)$ can be interpreted as the probability that an elliptic curve with prescribed torsion $G$ has multiplicative reduction at $p$.
For details and other examples, we recommend to see Corollaries \ref{p for mult} to \ref{cor: split nonsplit}.

\section{Local density and the moments of class numbers} \label{sec:prel}
\subsection{Model}
When we count the elliptic curves containing a torsion group $G$, we divide $G$ into the two cases.
Let 
\begin{equation*}
\mathcal{G}_{\leq 4} := \lcrc{\bZ/2\bZ, \bZ/3\bZ, \bZ/4\bZ, \bZ/2\bZ \times \bZ/2\bZ}
\end{equation*}
and $\mathcal{G}_{\geq 5}$ be the set of torsion groups of order $\geq 5$.
We often use $n$ and $2 \times 2m$ in place of $G = \bZ/n\bZ$ and $\bZ/2\bZ \times \bZ/2m\bZ$ to ease the notation. 

For each torsion subgroup, we should clarify the model we use.
When $G$ in $\mathcal{G}_{\leq 4}$, we recall the result of \cite[Theorem 1.1]{GT} shows that $E_{A,B} : y^2 = x^3 + Ax + B$ for $A, B \in \bZ$ has a $G$ as a torsion subgroup if and only if
\begin{equation*}
(A, B) = \Phi_G(a, b)
\end{equation*}
for some $a, b \in \bZ$,
where $\Phi_G = (f_G, g_G)$ for 
\begin{equation} \label{eqn:def fGgG}
\begin{array}{lll}
f_2(a,b) = a, & g_2(a,b) = b^3 + ab, \\
f_3(a,b) = 6ab + 27a^4, & g_3(a,b) = b^2 - 27a^6, \\ 
f_4(a,b) =  -3a^2 + 6ab^2 - 2b^4, & g_4(a,b) = (2a - b^2)(a^2 + 2ab^2 - b^4), \\
f_{2 \times 2}(a,b) =  -(a^2 + 3b^2)/4, & g_{2 \times 2}(a,b) = (b^3 - a^2b)/4.
\end{array}
\end{equation}
We recall that the set
\begin{equation*}
\cE(X) = \lcrc{(A, B) \in \bZ^2 : 
\begin{array}{cc}
|A| \leq X^{\frac{1}{3}}, |B| \leq X^{\frac{1}{2}}, 4A^3 + 27B^2 \neq 0, \\ 
\textrm{if $p^4$ divides $A$, then $p^6$ does not divide $B$. }
\end{array}
}
\end{equation*}
which parametrizes all elliptic curves $E_{A, B}$ whose height is less than $X$ and each isomorphism class appears only at once, by the minimality condition.
When $G$ is in $\mathcal{G}_{\leq 4}$, the set
\begin{align*}
	\cE_G(X) = \lcrc{(A, B) \in \cE(X) : (A, B) = \Phi_G(a, b) \textrm{ for some } a, b \in \bZ}
\end{align*}
parametrizes all elliptic curves with prescribed torsion subgroup $G$.

For $G$ in $\mathcal{G}_{\geq 5}$, we use Tate's normal form 
\begin{equation} \label{eqn:Tate}
   E(u,v) : y^2 + (1-v)xy - uy = x^3 - ux^2,
\end{equation}
which parametrizes all elliptic curves with prescribed torsion subgroup $G$ of order $\geq 4$. 
For each $G$, parameters $u$ and $v$ can be expressed as a rational function of one variable $t$.
It can be summarized as follow: (for example, \cite[Table 3]{Kub})
\begin{equation*}
	\begin{array}{|l|l|l|}
	\hline
	G & u(t) & v(t) \\ \hline
	4 & t & 0 \\ \hline
	5 & t & t \\ \hline
	6 & t + t^2 & t \\ \hline
	7 & t^3 - t^2 & t^2 - t \\ \hline
	8 & (2t-1)(t-1) & (2t-1)(t-1)/t \\ \hline
	9 & t^2(t-1)(t^2 - t + 1) & t^2(t-1) \\ \hline
	10 & t^3(2t-1)(t - 1)/(-t^2 + 3t - 1)^2 & t(2t-1)(t-1)/(-t^2 + 3t - 1) \\ \hline
	12 &(3t^2 - 3t + 1)(t - 2t^2)(2t - 2t^2 - 1)/(t-1)^4  & (3t^2 - 3t + 1)(t - 2t^2)/(t-1)^3 \\ \hline
	2 \times 4 & t^2 - 1/16 & 0 \\ \hline
	2 \times 6 & v(t) + v(t)^2 & (10 - 2t)/(t^2 - 9) \\ \hline
	2 \times 8 & (2t + 1)(8t^2 + 4t + 1)/(8t^2-1)^2 &  (2t + 1)(8t^2 + 4t + 1)/2t(4t + 1)(8t^2-1) \\ \hline
	\end{array}
\end{equation*}

For each torsion subgroup $G$, we first obtain an equation over $\bZ[t]$ by clearing the denominator of each coefficient. After that we take the usual coordinate change and obtain an equation of the form $y^2 = x^3 + f_G(t)x + g_G(t)$.
For $t=\frac{a}{b}$, the homogenization $f_G(a,b)=b^{\deg f}f_G(a/b)$ and $g_G(a,b)=b^{\deg g}g_G(a/b)$ of $f_G$ and $g_G$ and change of coordinate give
\begin{equation} \label{eqn:modelG5}
   y^2 = x^3 + f_G(a, b)x  + g_G(a, b).
\end{equation}
For simplicity, we use $f_{2 \times 4}(a, 4b)/8^4, g_{2 \times 4}(a, 4b)/8^6$ and 
$f_{2 \times 6}(a+3b, b), g_{2 \times 6}(a+3b, b)$ for $f_{2 \times 4}, g_{2\times 4}, f_{2 \times 6}$ and $g_{2 \times 6}$. 
One can check that $f_{2 \times 4}, g_{2 \times 4}$ and $f_{2 \times 6}, g_{2 \times 6}$ represent all isomorphism classes of elliptic curves with the corresponding torsion.
In Appendix \ref{sec:appendix}, the table for $f_G$ and $g_G$ is provided. 
For any torsion subgroup $G$ in $\mathcal{G}_{\geq 5}$ we have $3 \deg f_G = 2\deg g_G$,
and we define $d(G)$ as
\begin{equation*}
	3\deg f_G = 2 \deg g_G = 2d(G).
\end{equation*}

On the other hand, it is very crucial to recognize that the set
\begin{align*}
	\lcrc{(A, B) \in \cE(X) : (A, B) = (f_G, g_G)(a, b) \textrm{ for some } a, b \in \bZ}
\end{align*}
might not parametrize all isomorphism classes of elliptic curves with torsion subgroup $G$ in $\mathcal{G}_{\geq 5}$.
The reason is as follows: The Tate normal form parametrizes all isomorphism classes of elliptic curves with prescribed torsion, but to parametrize all the curves we need to consider all $t \in \bQ$, in other words all relatively prime integer pairs $(a, b)$.
But if there is an integer $e>1$ such that $e^4 \mid f_G(a, b)$ and $e^6 \mid g_G(a, b)$, then the minimal elliptic curve isomorphic to $E_{f_G(a, b), g_G(a, b)}$ may not appear in the above set since it is removed by the minimality condition.

Here the problem is that the map $(f_G, g_G)$ does not care the minimality condition.
Following \cite[Theorem 3.3.1]{CKV}, we define a \emph{defect} of $(a, b)$ to be
\begin{align*}
	 e(a, b) = e = \max_{\substack{e'^4 \mid f_G(a, b) \\ e'^6 \mid g_G(a, b)}} e' .
\end{align*}
We slightly modify the definition of $\Phi_G$ as follows:
\begin{align*}
	\Phi_G(a, b) = \lbrb{\frac{f_G(a, b)}{e^4}, \frac{g_G(a, b)}{e^6}}
\end{align*}
where $e$ is a defect of $(a, b)$. We remark that the image of $\Phi_G$ satisfies the minimality condition, so
\begin{align} \label{eqn:EG large}
	\cE_G(X) = \lcrc{(A, B) \in \cE(X) : (A, B) = \Phi_G(a, b) \textrm{ for relatively prime integers } a, b}
\end{align}
parametrizes all isomorphism classes of elliptic curves with torsion subgroup $G$.
We define a height of an integer pair $(A, B)$ by $\max(|A|^3, |B|^2)$ and
\begin{align*}
	M_G(X) = \lcrc{(a, b) \in \bZ^2 : (a, b) = 1, h(\Phi_G(a, b)) \leq X}.
\end{align*}
Hence $\Phi_G$ is a map from $\bZ^2$ to $\bZ^2$ when $G$ is in $\mathcal{G}_{\leq 4}$ 
and from $M_G(X)$ to $\bZ^2$ when $G$ is in $\mathcal{G}_{\geq 5}$.
Also, we define $M_G^e(X)$ as a set of elements of $M_G(X)$ with defect $e$.
Now we compute all defects for the torsion groups $G$, except $G = \bZ/2\bZ \times \bZ/6\bZ$ and $\bZ/2\bZ \times \bZ/8\bZ$.
\begin{lemma} \label{lem:defect}
Let $G$ be a group in $\mathcal{G}_{\geq 5} \setminus \{\bZ/2\bZ \times \bZ/6\bZ, \bZ/2\bZ \times \bZ/8\bZ \}$,
and let $e$ be the defect of a relatively prime integer pair $(a, b)$.
Then, the defect $e(a, b)$ is 1,2,3, or 6. 
Explicitly, we have \\
(i) $e$ has a prime divisor $2$ if and only if
\begin{itemize}
\item $G = \bZ/6\bZ$ and $(a, b) \equiv (1, 1) \pmod{2}$ or,
\item $G = \bZ/8\bZ$ and $(a, b) \equiv (1, 0) \pmod{2}$  or,
\item $G = \bZ/10\bZ$ and $(a, b) \equiv (1, 0) \pmod{2}$  or,
\item $G = \bZ/12\bZ$ and $(a, b) \equiv (1, 0) \pmod{2}$   or,
\item $G = \bZ/2\bZ \times \bZ/4\bZ$ and $(a, b) \equiv (1, 1) \pmod{2}$.
\end{itemize}
(ii) $e$ has a prime divisor $3$ if and only if
\begin{itemize}
\item $G  = \bZ/7\bZ$ and $(a, b) \equiv (1, 2)$ or $(2, 1) \pmod{3}$ or,
\item $G  = \bZ/9\bZ$ and $(a, b) \equiv (1, 2)$ or $(2, 1) \pmod{3}$ or,
\item $G = \bZ/12\bZ$, $a \not\equiv 0$, and $b \equiv 0 \pmod{3}$.
\end{itemize}
\end{lemma}
\begin{proof}
%
%

By the argument \cite[p.17]{CKV},
the defect $e$  is a divisor of the least common multiplier of the two resultants 
$\mathrm{Res}(f_G(a, 1), g_G(a, 1))$ and $\mathrm{Res}(f_G(1, b), g_G(1, b))$.
Sagemath \cite{Sag} gives
\begin{equation*}
	\begin{array}{|c|c|c|c|} \hline
	G & \textrm{l.c.m of resultants} & G & \textrm{l.c.m of resultants} \\ \hline
	\bZ/5\bZ & 2^{16}3^{36}5 & \bZ/10\bZ &  2^{72} 3^{108} 5^{3} \\ \hline
	\bZ/6\bZ & -2^{24}3^{39} & \bZ/12\bZ & 2^{96} 3^{156} \\ \hline
	\bZ/7\bZ & -2^{32}3^{72}7 & \bZ/2\bZ \times \bZ/4\bZ & 2^{24}3^{36} \\ \hline
	\bZ/8\bZ & 2^{48}3^{72} &  \bZ/2\bZ \times \bZ/6\bZ &  2^{192}3^{78} \\ \hline
	\bZ/9\bZ & -2^{48}3^{117} &  \bZ/2\bZ \times \bZ/8\bZ & 2^{576}3^{144} \\ \hline
	\end{array}
\end{equation*} 
Hence any prime divisor of $e$ should divide $6|G|$.


First, we find all the pairs pairs $(a, b) \in (\bZ/p^6\bZ)^2$ such that $a, b$ are relatively prime to $p$ and $p^4 \mid f_G(a, b)$ and $p^6 \mid g_G(a, b)$ for each prime divisor $p$ of $6|G|$. Then, there is no such pair $(a,b)$ except for the following 4 cases:
\begin{itemize}
\item when $G = \bZ/6\bZ$ and $(a, b) \equiv (1, 1) \pmod{2}$, $2$ exactly divides $e$ 
\item when $G  = \bZ/7\bZ$ and $(a, b) \equiv (1, 2)$ or $(2, 1) \pmod{3}$, $3$ exactly divides $e$.
\item when $G  = \bZ/9\bZ$ and $(a, b) \equiv (1, 2)$ or $(2, 1) \pmod{3}$, $3$ exactly divides $e$.
\item when $G = \bZ/2\bZ \times \bZ/4\bZ$ and $(a, b) \equiv (1, 1) \pmod{2}$, $2$ exactly divides $e$.
\end{itemize} 

Now, we consider the pairs $(a,b)$ for which only one of $a$ and $b$ is a multiple of $p$. When $G = \bZ/5\bZ$ and $p = 2$, if only one of the $a$ or $b$ is divided by $2$ then $2^4$ does not divide $f_5(a, b)$ because the coefficients of $a^4$ and $b^4$ are not divided by $2^4$. 
Hence we can conclude that $2$ does not divide the defect $e$ for arbitrary $(a, b)$.
Considering the coefficients of $f_G$ and $g_G$ (see Appendix \ref{sec:appendix}),
the same argument shows that the possible prime divisors of defect are (with the previous four cases)
\begin{itemize}
\item when $G = \bZ/8\bZ$ and $(a, b) \equiv (1, 0) \pmod{2}$, $2$ divides $e$.
\item when $G = \bZ/10\bZ$ and $(a, b) \equiv (1, 0) \pmod{2}$, $2$ divides $e$.
\item when $G = \bZ/12\bZ$ and $(a, b) \equiv (1, 0) \pmod{2}$, $2$ divides $e$.
\item when $G = \bZ/12\bZ$, $a \not\equiv 0$, and $b \equiv 0 \pmod{3}$, $3$ divides $e$.
\end{itemize}

For the first three cases we can check that there is no $(a, b) \in (\bZ/2^6\bZ)^2$ such that $2^5 \mid f_G(a, b)$ and $2^6 \mid g_G(a, b)$, which implies that $2^2 \nmid e$.
Similarly for the fourth case, we can check that there is no $(a, b) \in (\bZ/3^6\bZ)^2$ such that $3^6 \mid f_G(a, b)$ and $3^6 \mid g_G(a, b)$.
For cross-check, we refer Appendix \ref{sec:appendix} for our $f_G$ and $g_G$.
\end{proof}

\begin{remark} \label{rem:defect}
We note that one may calculate the defects for the remaining two groups by following the proof of Lemma \ref{lem:defect}. For example for $G = \bZ/2\bZ \times \bZ/6\bZ$ when $a \equiv 0 \pmod{4}$ and $b \not\equiv 0 \pmod{2}$, $2^2$ exactly divides $e(a,b)$ and when $a \equiv 2 \pmod{4}$ and $b \equiv 1 \pmod{2}$, $2^3$ divides $e(a,b)$. 
It seems that the defect is $2^4$ but to check it we need more computing power. 
Instead, we omit $G =  \bZ/2\bZ \times \bZ/6\bZ$ and $\bZ/2\bZ \times \bZ/8\bZ$ cases.
\end{remark}

\subsection{Weights for local conditions} \label{subsec:local density}
We define a weight for a local condition as the number of preimages of $(f_G, g_G)$ modulo $p$.
\begin{definition}
For a prime $p \geq 5$, and a pair $J \in (\bZ/p\bZ)^2$, let 
$W_{G,J}$ be the set of pairs $I \in (\bZ/p\bZ)^2$ with 
$(f_G, g_G)(I) \equiv J$ modulo $p$.
\end{definition}
For a given $J$, $|W_{G, J}|$ is morally a weight to determine the number of elliptic curves $E$ with mod $p$ reduction $E_{J}$ and $E(\bQ)_{\textrm{tor}} \geq G$. By the definition of $W_{G,J}$, the identity 
\begin{equation*}
\sum_{J \in (\bZ/p\bZ)^2} |W_{G,J} | = p^{2}
\end{equation*}
follows directly. 
\begin{proposition} \label{prop:sum of WGI}
	For a prime $p \geq 5$, the sums of $|W_{G, J}|$ over  $J = (A,  B) \in \bF_p^2$ satisfying $4 A^3 + 27B^2 \equiv 0 \pmod{p}$ are summarized as follows:
	\begin{equation*}
	\begin{array}{|l|l|l|l|}
	\hline
	G & p &  \sum |W_{G,J}|  \\ \hline
	\bZ/2\bZ & \cdot & 2p - 1 \\ \hline
	\bZ/3\bZ & \cdot & 2p - 1 \\ \hline
	\bZ/4\bZ & \cdot & 3p - 2 \\ \hline
	\bZ/2\bZ \times \bZ/2\bZ & \cdot & 3p-2 \\ \hline
	\bZ/5\bZ & \pm 1 \pmod{5} & 4p-3 \\ \hline
	\bZ/5\bZ & \pm 2  \pmod{5} & 2p-1 \\ \hline
	\bZ/6\bZ&  \cdot  & 4p-3 \\ \hline
	\bZ/7\bZ & \gamma_7 \in (\bF_p[\sqrt{-3}]^\times)^3 & 6p-5 \\ \hline
	\bZ/7\bZ & \gamma_7 \not\in (\bF_p[\sqrt{-3}]^\times)^3 & 3p-2 \\ \hline
	\bZ/8\bZ & \pm 1 \pmod{8} & 6p - 5 \\ \hline
	\bZ/8\bZ & \pm 3 \pmod{8} & 4p-3 \\ \hline
	& & \\ \hline
	& & \\ \hline
	& & \\ \hline
	\end{array}
	\begin{array}{|l|l|l|l|}
	\hline
		G & p &  \sum |W_{G,J}|  \\ \hline
	\bZ/9\bZ & 1 \pmod{3},  \gamma_9 \in (\bF_p^\times)^3  & 8p-7 \\ \hline
	\bZ/9\bZ  & 1 \pmod{3},   \gamma_9 \not\in (\bF_p^\times)^3  & 5p-4 \\ \hline
	\bZ/9\bZ & 2 \pmod{3},   \gamma_9 \in (\bF_p[\sqrt{-3}]^\times)^3  & 6p-5 \\ \hline
	\bZ/9\bZ & 2 \pmod{3},  \gamma_9 \not\in (\bF_p[\sqrt{-3}]^\times)^3  & 3p-2 \\ \hline
	\bZ/10\bZ &  \pm 1 \pmod{5} & 8p-7 \\ \hline
	\bZ/10\bZ &  \pm 2 \pmod{5} & 4p-3 \\ \hline
	\bZ/12\bZ & 1 \pmod{12}  & 10p-9 \\ \hline
	\bZ/12\bZ & 5, 7, 11 \pmod{12} & 6p-5 \\ \hline
	\bZ/2\bZ \times \bZ/4\bZ & \cdot & 4p-3 \\ \hline
	\bZ/2\bZ \times \bZ/6\bZ & \cdot & 6p-5  \\ \hline
	\bZ/2\bZ \times \bZ/8\bZ & 1 \pmod{8}, \geq 11 & 10p-9  \\ \hline
	\bZ/2\bZ \times \bZ/8\bZ & 7 \pmod{8}, \geq 11 & 8p-7  \\ \hline
	\bZ/2\bZ \times \bZ/8\bZ & 5 \pmod{8}, \geq 11 & 6p-5  \\ \hline
	\bZ/2\bZ \times \bZ/8\bZ & 3 \pmod{8}, \geq 11 & 4p-3  \\ \hline
	\end{array}
\end{equation*}
where $\gamma_7 = 4(637 + 147\sqrt{-3})$ and $\gamma_9 = 4(-9 \pm 3\sqrt{-3})$.
Here $\cdot$ means that there is no condition on $p$.
Furthermore, we have
\begin{align} \label{eqn:bias split nonsplit for G=3}
\sum_{\alpha=a^2 \in (\bZ/p\bZ)^\times} |W_{3,(-3\alpha^2,2\alpha^3)}| = 
\begin{cases}  2(p-1) &  \mbox{ for $p\equiv 1$ mod $12$, }  \\ 
(p-1) &  \mbox{ for $p\equiv 5$ or $11$ mod $12$, } \\
0 &  \mbox{ for $p\equiv 7$ mod $12$. }
\end{cases}
\end{align}
\end{proposition}
\begin{proof}
We note that for $p\geq5$, the pair $I=(0,0)$ in $(\bZ/p\bZ)^2$ is the only pair such that $(f_G,g_G)(I)\equiv (0,0) \pmod{p}.$
For the groups $G$ with order $\leq 4$, one can directly check it. We show the case of $G=\bZ/3\bZ$.  We parametrize $(A, B)$ satisfying $4A^3 + 27B^2 \equiv 0$ by $(-3 \alpha^2, 2\alpha^3)$ for $\alpha \in \bZ/p\bZ$. 
Directly solving the equations $\Phi_G(a, b) = (A, B)$, 
we know that $|W_{3,  (A, B)}|$ is equal to the number of distinct zeros of the polynomial
\begin{equation*}
h(x) = h_{A, B}(x) = 3^5 \cdot x^8 + 2 \cdot 3^3 \cdot A \cdot x^4 + 2^2 \cdot 3^2 \cdot B \cdot x^2 - A^2,
\end{equation*}
when $A \not\equiv 0.$
Since $h_{-3\alpha^2, 2\alpha^3}(x)$ is factored into 
\begin{align*}
3^5\left( x^2- \frac{\alpha}{3}\right)^3 \left(x^2	+\alpha \right),
\end{align*}
the number of distinct zeros of $h_{-3\alpha^2, 2\alpha^3}(x)$ is 4 if $-\alpha$ and $\alpha/3$ are both quadratic residues modulo $p$, $2$ if either $-\alpha$ or $\alpha/3$ is a quadratic residue, and 0 if neither $-\alpha$ nor $\alpha/3$ is a quadratic residue. From this observation, it is easy to see that the sum of distinct zeros of  $h_{-3\alpha^2, 2\alpha^3}(x)$ over $\alpha \in \bZ/p\bZ$ is $2(p-1)+1=2p-1$. 
Furthermore, if $-1$ and $3$ are quadratic residue which is equivalent to $p \equiv 1 \pmod{12}$, 
then the sum of $|W_{3, (-3\alpha^2, 2\alpha^3)}|$ over quadratic residues $\alpha$ in $\bZ/p\bZ$
is equal to the sum of $|W_{3, (-3\alpha^2, 2\alpha^3)}|$ over all non-zero residues $\alpha$ in $\bZ/p\bZ$.
Hence, we obtain the $G = \bZ/3\bZ$ row and the equation (\ref{eqn:bias split nonsplit for G=3}).

Let $(a,b)$ be a pair such that $4f_G(a, b)^3 + 27g_G(a, b)^2 \equiv 0 \pmod{p}$. Then, this $(a,b)$ determines $\alpha$ in $\bZ/p\bZ$ satisfying $(f_G, g_G)(a, b) \equiv (-3 \alpha^2, 2 \alpha^3)$.  Hence, if we find such all pairs $(a,b)$ (including the $(0,0)$ pair), then the number of the pairs is the sum we want to know.
We will consider the discriminant of $E_{(f_G(a, b), g_G(a, b))}$, instead of $(f_G, g_G)(a, b)$.
	
	Let $\Delta_G(a, b)$ be the discriminant of $E_{(f_G, g_G)(a, b)}$. Then, we have 
	\begin{equation*}
	\begin{array}{|l|l|l|}
	\hline
	G &  \Delta_G(a,b)   \\ \hline
	\bZ/5\bZ & 2^{12} 3^{12} a^5 b^5 (a^2 - 11ab - b^2)  \\ \hline
	\bZ/6\bZ & -2^8 3^{12} a^6 b^2 (9a + b)  (a + b)^3  \\ \hline
	\bZ/7\bZ & 2^{12} 3^{12} a^7b^7(a - b)^7 (a^3 - 8a^2b + 5ab^2 + b^3) \\ \hline
	\bZ/8\bZ & 2^{12} 3^{12} a^8 b^2 (-2a + b)^4 (-a + b)^8  (8a^2 - 8ab + b^2) \\ \hline
	\bZ/9\bZ &   2^{12} 3^{12} a^9 b^9  (a - b)^9 (a^2 - ab + b^2)^3(a^3 - 6a^2b + 3ab^2 + b^3)  \\ \hline
	\bZ/10\bZ & 2^{12}   3^{12}  b^5   (-2 a + b)^5   (-a + b)^{10}   a^{10}   (-4 a^2 + 2 a b + b^2)   (a^2 - 3 a b + b^2)^2 \\ \hline
	\bZ/12\bZ & 2^{12}   3^{12}   b^2   (-2 a + b)^6   (-a + b)^{12}   a^{12}   (6 a^2 - 6 a b + b^2)   (2 a^2 - 2 a b + b^2)^3   (3 a^2 - 3 a b + b^2)^4 \\ \hline
	\bZ/2\bZ \times \bZ/4\bZ & 2^8   3^{12}  b^2   a^2   (a - b)^4   (a + b)^4  \\ \hline
	\bZ/2\bZ \times \bZ/6\bZ & 2^{18}   3^{12} a^2   (a - 6 b)^2   (a + 6 b)^2   b^6   (a - 2 b)^6   (a + 2 b)^6 \\ \hline
	\bZ/2\bZ \times \bZ/8\bZ &2^{20}   3^{12}  b^8   a^8   (2 a + b)^8   (4 a + b)^8   (8 a^2 - b^2)^2   (8 a^2 + 8 a b + b^2)^2   (8 a^2 + 4 a b + b^2)^4 \\ \hline
	\end{array}
\end{equation*}
		First, let's treat the cases where $\Delta_G(a, b)$ is a product of linear polynomials and quadratic polynomials.
	For example, consider 
	\begin{equation*}
	\Delta_8(a, b) = 2^{12} 3^{12} a^8 b^2 (-2a + b)^4 (-a + b)^8  (8a^2 - 8ab + b^2).
	\end{equation*}
	So in this case we have four types of $(a,b)$ satisfies the condition which are $a = 0, b=0, 2a/b = 1, a/b = 1$, and $a/b$ is a zero of the quadratic polynomial $8t^2 - 8t + 1$. The first four cases give $(p-1)$-pairs, and the quadratic polynomial has a zero in $\bF_p$ when $p \equiv 1, 7 \pmod{8}$. Since the value of $8t^2 - 8t + 1$ at $t = 1, 1/2$ is $\pm 1$, there is no overlap among those solutions. Hence we verified the case of $G = \bZ/8\bZ$. The other cases can be handled similarly.

%
%
%

Now, let's verify the cases where $\Delta_G(a,b)$ contains a cubic polynomial.  For this purpose, we need the following lemma.

\begin{lemma} \label{lem:Dic}
	Let $f(t) = t^3 + at + b$ be a polynomial over $\bF_p$ with the discriminant $\Delta = (-4a^3 - 27b^2)$. The number of zeros (without multiplicity) of $f(t)$ is \\
	(1) zero if and only if $\Delta = 81\mu^2$ is square and $(-b + \mu \sqrt{-3})/2$ is not cube in the field $\bF_p[\sqrt{-3}]$. \\
	(2) one if and only if $\Delta$ is non-square. \\
	(3) two if and only if $\Delta$ is zero. \\
	(4) three for other cases.
\end{lemma}
\begin{proof}
	The first, second and fourth statements are shown in \cite{Dic} and the third one follows from the fact that a monic cubic which has two zeros and has no degree two term is parametrized by $(t-2a)(t+a)^2$.
\end{proof}

When $G = \bZ/7\bZ$, there is a polynomial $(a^3 - 8a^2b + 5ab^2 + b^3)$ in $\Delta_G(a, b)$.
We obtain $(t^3 - \frac{49}{3}t - \frac{637}{27})$ by change of coordinate.
In this case  the discriminant of this polynomial is $2401 = 7^4$, so when $p > 7$ then the number of zeros is one of $0$ or  $3$.
Also, $\mu =49/9$ and the number of zeros is determined by $\frac{1}{2}(-\frac{637}{27} + \frac{49}{9}\sqrt{-3})$ which is equal to $4(-637 + 147\sqrt{-3})$ up to a cube. We note that the 0 or 1 is not a solution of the given polynomial which means that there is no overlap, so we obtain the row for $G = \bZ/7\bZ$.
When $G = \bZ/9\bZ$, we can prove it similarly. 
%
\end{proof}

We need to prove some elementary but not simple properties of $\Phi_G$.
We put
\begin{equation*}
\begin{array}{|c|c|c|c|c|c|c|}
\hline
G & \lcrc{0}  & \bZ/2\bZ & \bZ/3\bZ & \bZ/4\bZ &\bZ/2\bZ \times \bZ/2\bZ & G  \textrm{ in } \mathcal{G}_{\geq 5} \\
\hline
e(G) &2  & 3 & 4 & 6 & 6 & 2d(G)\\
\hline
\end{array}
\end{equation*}

\begin{lemma} \label{lem:preimage rG}
For $G \in \cG_{\leq 4}$, there is a positive integer $r(G)$ such that the number of the preimages of $\Phi_G$ is $r(G)$ except $O(X^{\frac{1}{e(G)}})$-points.
\end{lemma}
\begin{proof}
 The cases $G = \bZ/2\bZ, \bZ/3\bZ$ are essentially in \cite[Lemma 5.5]{HS} with $r(\bZ/2\bZ) = 1$ and $r(\bZ/3\bZ) = 2$.
 For $G = \bZ/4\bZ$, assume that there are $(a',b') \neq (a, \pm b)  $ such that
 \begin{equation*}
  (-3a^2 + 6ab^2 - 2b^4, (2a - b^2)(a^2 + 2ab^2 - b^4) ) =  (-3a'^2 + 6a'b'^2 - 2b'^4, (2a' - b'^2)(a'^2 + 2a'b'^2 - b'^4) ).
 \end{equation*}
 The elliptic curve $E_{\Phi_4(a,b)}$ has a 4-torsion point $(a, b(-b^2+3a))$. 
 Since an elliptic curve over rational numbers does not have $\bZ/4\bZ \times \bZ/4\bZ$ as a subgroup, $(a,b)$ and $(a', b')$ also satisfy
 \begin{equation*}
 (a, b(-b^2 + 3a)) =  (a',  \pm b'(-b'^2 + 3a')).
 \end{equation*}
 If $b^2 \neq b'^2$, we obtain $bb' = 0$. Without loss of generality we may assume that $b' = 0$, then we have $3a = b^2$. Then, a 4-torsion point  $(a, b(-b^2+3a))$ is a 2-torsion point, which is a contradiction.
 We note that $r(\bZ/4\bZ) = 2$.
 
 Let $G = \bZ/2\bZ \times \bZ/2\bZ$. By a similar argument, we need to count $(A,B)$ such that 
 \begin{equation*}
 \left(A,B\right) = \left(-\frac{a^2 + 3b^2}{4}, \frac{b^3 - ba^2}{4} \right) = \left(-\frac{a'^2 + 3b'^2}{4}, \frac{b'^3 - b'a'^2}{4} \right), 
 \end{equation*}
 and
 \begin{equation*}
 \left\{ \frac{a+b}{2}, \frac{b-a}{2}, -b \right\} =  \lcrc{\frac{a'+b'}{2}, \frac{b'-a'}{2}, -b' }.
 \end{equation*}
We note that since $A$ and $B$ are integers, $a$ and $b$ should have the same parity.
The set equality allows the identity of $y$-coordinates on the first equation, and it holds if and only if one of the following six linear systems
 \begin{equation*}
  \left( \begin{array}{c} a' \\ b' \end{array} \right)= A_i \left( \begin{array}{c} a \\ b \end{array} \right),   \end{equation*}
  for $A_0 = I$, and
 \begin{align*}
 A_1 = \lbrb{\begin{array}{cc} 
 -1 & 0 \\ 0 &1
 \end{array}
 },
A_2 =\left( \begin{array}{rcrc}
 \frac{1}{2} & -\frac{3}{2} \\
 -\frac{1}{2} & -\frac{1}{2}
\end{array} \right),
A_3=\left( \begin{array}{rcrc}
 -\frac{1}{2} & \frac{3}{2} \\
 -\frac{1}{2} & -\frac{1}{2}
\end{array} \right),
A_4=\left( \begin{array}{rcrc}
 \frac{1}{2} & \frac{3}{2} \\
 \frac{1}{2} & -\frac{1}{2}
\end{array} \right),
A_5=\left( \begin{array}{rcrc}
 -\frac{1}{2} & -\frac{3}{2} \\
 \frac{1}{2} & -\frac{1}{2}
\end{array} \right). 
\end{align*}
Consequently, for $(a,b)$ satisfying $a\equiv b \pmod{2}$, the (not necessarily distinct) six points
\begin{align*}
(a,b),(-a,b),\left(\frac{a-3b}{2}, \frac{-a-b}{2}\right), \left(\frac{-a+3b}{2}, \frac{-a-b}{2}\right),\left(\frac{a+3b}{2}, \frac{a-b}{2}\right),\text{ and } \left(\frac{-a-3b}{2},  \frac{a-b}{2}\right)
\end{align*}
corresponds to the same $(A,B)$. We find a domain where the representatives for the above (not necessarily distinct) six points.
We claim that the following set 
\begin{align*}
X = \left\{ (a,b) \in \mathbb{Z} \times \mathbb{Z} :  a \geq 0, b\geq \frac{a}{3}, a\equiv b \text{ mod }2 \right\}
\end{align*}
is the collection of all the representatives of the above (not necessarily distinct) six points. 
On the other hand, the number of points such that the number of their preimages is strictly less than six is $O(X^{\frac{1}{6}})$.
Hence, we obtain the result with $r(\bZ/2\bZ \times \bZ/2\bZ) = 6$.
\end{proof}

For $G \in \cG_{\geq 5}$, we can prove analogous statement by using the argument of \cite{CKV}.
\begin{lemma} \label{assumption}
	For $G \in \cG_{\geq 5}$, there is an integer $r(G)$ such that the preimages of $\Phi_G$ is $r(G)$ except $O(X^{\frac{1}{e(G)}})$-points.
\end{lemma}
\begin{proof}
Essentially it is proved in the proof of \cite[Theorem 3.3.1]{CKV}, so here we give a sketch.
For a $G \in \cG_{\geq 5}$ and corresponding congruence subgroup $\Gamma$, there is a bijection between the set of $\bQ$-isomorphism classes of elliptic curve with $\Gamma$-structure and rational points of the modular curve $Y_\Gamma$ (see \cite[Proposition 3.1.1]{CKV}).
By choosing a coordinate that defines an embedding $Y_\Gamma \to \mathbb{A}_\bQ^1$, the proof of \cite[Theorem 3.3.1]{CKV} gives a bijection from $Y_\Gamma(\bQ)$ to the set
\begin{align*}
	\lcrc{(a, b) \in \bZ^2 : |f_G(a, b)| \leq X^{\frac{1}{3}}, |g_G(a, b)| \leq X^{\frac{1}{2}}, (a, b) = 1}.
\end{align*}
Now, the natural map from the elliptic curve with $\Gamma$-structure to the set of elliptic curves which has $\Gamma$-structure is $r(G)$-to-one map by \cite[Lemma 3.1.8]{CKV} except negligible set comes from the curves with $\Gamma'$-structure for $\Gamma' \subset \Gamma$ and curves whose $j$-invariants is 0 or 1728.
\end{proof}

Let $J$ be an element in $(\bZ/p\bZ)^2$ such that $E_J$ is an elliptic curve and $W_{G, J}$ is non-empty. 
Then for each $(a, b) \in W_{G, J}$ we have a change of coordinate from $E(u, v)$ whose equation is (\ref{eqn:Tate}) to $E_J : y^2 = x^3 + f_G(a, b)x + g_G(a, b)$ which is defined earlier. 
Since the change of coordinate gives an isomorphism between the groups of $\bF_p$-points, the image of $(0,0)$ of $E(u, v)$ also goes to a torsion point of maximal order.
When $G$ is cyclic,
it defines a map $\Psi_{G, J} : W_{G, J} \to E_J(\bF_p)$ whose image is in the set of points of maximal order in $G$.

\begin{lemma} \label{prop:value of WGI}
Let $G \in \mathcal{G}_{\leq 4}$, $G=\bZ/5\bZ,$ $\bZ/6\bZ$, or $\bZ/2\bZ \times \bZ/4\bZ$, and $J \in (\bZ/p\bZ)^2$ for $p \geq 5$ such that $E_J$ is an elliptic curve. 
Then, 
\begin{align*}
E_J(\bF_p) \geq G \qquad 
\textrm{if and only if} \qquad J = \Phi_G(a,b)
\end{align*}
for some $(a, b)$. Furthermore, $|W_{G, J}|$ is the number of embedding of $G$ into $E_J(\bF_p)$.
\end{lemma}

\begin{proof}
When $G \in \cG_{\leq 4}$, for the if and only if part we will use the computation of \cite{GT}.
For example when $G = \bZ/4\bZ$, 
 assume that $(x_1,w_1)$ is a point of order 4 of an elliptic curve $E_{A,B}/\bF_p$. 
From the computation of the first coordinate of $[3](x_1, w_1) = (x_1, -w_1)$ for $w_1 \neq 0$, we have
\begin{equation*}
B = \frac{1}{4}\lbrb{5x_1^3 - Ax_1 \pm \sqrt{(3x_1^2 - 2A)(A+3x_1^2)^2}  }.
\end{equation*}
Hence, $B$ is in $\bF_p$ if and only if there exists $x_2 \in \bF_p$ such that $3x_1^2 - 2A = x_2^2$.
The computation of the second coordinate gives that $w_1^2 = (3x_1 - x_2)(x_2 + 3x_1)^2/8$,
so we have $x_2 \neq 3x_1$ and there exists $x_3 \in \bF_p^\times$ such that $x_3^2 = (3x_1-x_2)/2$.
By the change of variables $a = x_1$ and $b = x_3$, we have $(A,B) = \Phi_4(a,b)$ with
points of order 4, $(x_1, w_1) = (a,\pm b(-b^2 + 3a))$.
For the converse, we know that $(a,\pm b(-b^2 + 3a))$ are points of order $4$ of $E_{\Phi_4(a,b)}$.
The other cases with order $\leq 4$ can be proved similarly, but we remark that the first equation of \cite[p. 92]{GT} should be
\begin{equation*}
(-z_2,0), \quad  \lbrb{\frac{1}{2}(z_2 \pm \sqrt{z_2^2 - 4z_1}), 0};.
\end{equation*}

Now we prove the second statement when $G \in G_{\leq 4}$.
When $G$ is cyclic, it suffices to prove that $\Psi_{G, J}$ is bijective.
For example $G = \bZ/4\bZ$, we note that for $J = \Phi_4(a,b)$,  the $2$-tuple $(a, -b)$ also corresponds to the same $J$ but they induce
the two points of order 4, $(a,\pm b(-b^2 + 3a))$.
Therefore, for each points of order $4$ in $E_{A,B}(\bF_p)$ there is an $(a, b) \in W_{4, (A,B)}$.
For the converse, let $\Phi_4(a', b') = (A,B)$ and $(a', b'( - b'^2 + 3a')) = (a, b(-b^2 + 3a)).$
If $b \neq b'$ then we have $a = a'$ and $bb'=0$ which implies that the one of the points $(a', b'( - b'^2 + 3a'))$ and $(a, b(-b^2 + 3a))$ is of order 2. The cases $G = \bZ/2\bZ$ and $\bZ/3\bZ$ can be dealt similarly.

We treat the case $G = \bZ/2\bZ \times \bZ/2\bZ$ separately.
We recall that $E_J$ has a $\bZ/2\bZ \times \bZ/2\bZ$ if and only if $J = (f_{2\times 2}(a, b), g_{2\times 2}(a, b))$. Hence if $E_J$ does not have full 2-torsion, then $|W_{2 \times 2, J}|$ should be zero.
	It is easily deduced that  if $E_J$ does not have full 2-torsion, then $W_{2 \times 2, J}$ should be empty. If $E_{J}$ has the full 2-torsions,  then $b^3 + Ab + B \equiv 0 \pmod{p}$ has three zeros and $A = f_{2\times 2}(a, b) =  -(a^2 + 3b^2)/4$ for some $a$.
   This $a$ is not zero, since if so then $f_{2 \times 2}(0, b) = -3b^2/4$ and $g_{2 \times 2}(0, b) = b^3/4$ so $4f_{2 \times 2}(0, b)^3 + 27 g_{2 \times 2}(0, b)^2 \equiv 0 \pmod{p}$.
   Hence, there are exactly six $(a, b)$ such that 
   \begin{align*}
   b^3 + Ab + B \equiv 0, \qquad 4A \equiv -(a^2 + 3b^2) \pmod{p}.
   \end{align*}
	Since this equation is equivalent to a system of equation $J = (f_{2 \times 2}(a, b), g_{2 \times 2}(a, b))$, we can conclude that if $E_J(\bF_p) \geq \bZ/2\bZ \times \bZ/2\bZ$ then $|W_{2 \times 2, J}| = 6$.

When $G = \bZ/5\bZ$ by taking homogenization and computing the multiples of the points, we know that the points
\begin{equation*}
	(3a^2- 18ab + 3b^2, \pm 108ab^2), \qquad (3a^2 + 18ab + 3b^2, \pm 108a^2b)
\end{equation*}
are points with order 5 of $E_{J}$ where $J  = \Phi_G(a,b)$. 
When $(a, b)$ gives one of above four points, then other three come from $(-b, a), (b, -a)$ and $(-a, -b)$.


We claim that the four pairs are all the pairs $(c,d)$ such that $\Phi_5(c,d)=J$ and 
\begin{equation*}
	3a^2 - 18ab + 3b^2 = 3c^2 -18cd + 3d^2, 
	\qquad 3a^2 + 18ab + 3b^2 = 3c^2 + 18cd + 3d^2,
\end{equation*} or
\begin{equation*}
	3a^2 - 18ab + 3b^2 = 3c^2 +18cd + 3d^2, 
	\qquad 3a^2 + 18ab + 3b^2 = 3c^2 - 18cd + 3d^2.
\end{equation*}
Both systems do not generate new pairs.
Therefore, $\Psi_5$ is injective and $|W_{G, J}|$ is less than or equal to the number of points of order 5 in $E_{J}$.


Let $P$ be a point of order $5$ in $E_{J}(\mathbb{F}_p)$ and $E_{J} : y^2=x^3+Ax+B$. 
Let $x_1$ and $x_2$ be the $x$-coordinates of $P$ and $2P$. Then by the duplication formula, we have
\begin{align*}
\frac{x_1^4-2Ax_1^2-8Bx_1+A^2}{4(x_1^3+Ax_1+B)}=x_2, \quad \frac{x_2^4-2Ax_2^2-8Bx_2+A^2}{4(x_2^3+Ax_2+B)}=x_1.
\end{align*}
From the identity above, we can see that $2x_1+x_2$ and $x_1+2x_2$ are squares in $\mathbb{F}_p$.
Let  $\sqrt{2x_1+x_2}$ and $\sqrt{x_1+2x_2}$ be one of the square roots of $2x_1+x_2$ and $x_1+2x_2$ respectively. Then, by putting 
\begin{align*}
a=\frac{\sqrt{2x_1+x_2} + \sqrt{x_1+2x_2}}{6}, \quad b=\frac{\sqrt{2x_1+x_2} - \sqrt{x_1+2x_2}}{6},
\end{align*}
we have
\begin{align*}
x_1=3a^2+18ab+3b^2, \quad x_2=3a^2-18ab+3b^2,
\end{align*}
and one can check easily that $A=f_G(a,b)$ and $B=g_G(a,b)$. Hence for the point  $P$ of order 5, we found $(a,b) \in W_{J}$ such that $P = (3a^2 - 18ab + 3b^2, 108ab^2)$ which shows the surjectivity of $\Psi_5$.
%

	As we did in the $\bZ/5\bZ$-case, we can show that
	\begin{equation*}
	(-9a^2 -18ab + 3b^2,  \pm(108a^2b + 108ab^2))
	\end{equation*}	 
	are points of order 6 of elliptic curve $E_J$ where $J = \Phi_6(a, b)$ for some $a, b \in \bF_p$.
	Since $(-a, -b)$ also gives a same points, we have a map $\Psi_6$ from $W_{6, J}$ to the points of order 6 of $E_J(\bF_p)$.
	
	We claim that $(c,d)=\pm(a,b)$ are all the pair such that $\Phi_6(a,b)=\Phi_6(c, d)$ and $\Psi_6(a, b) = \Psi_6(c, d)$. Considering the $x$-coordinates of multiplies of the point $\Psi_6(a,b)$ we have
	\begin{align*}
	 -9a^2 - 18ab + 3b^2 &= -9c^2 - 18cd + 3d^2, \\
	 27a^2+18ab + 3b^2 &= 27c^2 + 18cd + 3d^2, \\  
	 -9a^2+18ab+3b^2 &= -9c^2 + 18cd + 3d^2,
	\end{align*}
	since the $x$-coordinate of $2P$ and $3P$ is $27a^2+18ab + 3b^2, -9a^2+18ab+3b^2$, respectively. This system does not generate a new pair.
Therefore $\Psi_6$ is injective.

	Let $P := (x_1, y_1)$ be a point of order 6 of $E_J(\bF_p)$ and let $(x_2, y_2) := 2P$, and $(x_3, 0):= 3P$.
	By the duplication formula, we know that $2x_1+x_2$ is square. 
	Since $2P$ is a point of order 3, then $J = \Phi_3(a_3, b_3)$ for some $a_3, b_3 \in \bF_p$ and $(3a_3^2,  \pm(9a_3^3 + b_3))$ are 3-torsion point of $E_J$.
	Especially, we note that $x_2/3$ is square in $\bF_p$.
	Now, we define
	\begin{align*}
		a := \frac{3\sqrt{x_2/3} + \sqrt{2x_1 + x_2} }{12}, \qquad
		b := \frac{\sqrt{x_2/3} - \sqrt{2x_1 + x_2}}{4}.
	\end{align*}
	Both are in $\bF_p$ and  we have $x_2 = 27a^2 + 18ab + 3b^2$, $x_1 = -9a^2 - 18ab + 3b^2$.
	Using the result on 3-torsion case, one can easily check that $(f_6(a, b), g_6(a, b)) = (A, B)$.
	
Let $G = \bZ/2\bZ \times \bZ/4\bZ$ and assume that $J = \Phi_{2 \times 4}(a, b)$ for some $a,b$. The claim is 
\begin{equation*}
|W_{ 2 \times 4, J}|  = \left\{
\begin{array}{cl}
24 & \textrm{if } E_J(\bF_p)[4] \cong \bZ/4\bZ \times \bZ/4\bZ, \\
8 & \textrm{if } E_J(\bF_p)[4] \cong \bZ/2\bZ \times \bZ/4\bZ. \\
\end{array}
\right.
\end{equation*}
Considering 6 systems deduced by $\Phi_{2\times 4}(a, b) = \Phi_{2\times4}(c, d)$, 
 we can see that $\left| W_{2 \times 4, J} \right|$ is at least $8$ and it should be exactly $8$ if either $p\equiv 3$ (mod $4$) or $p\equiv 1$ (mod $4$) and $ab$ is  non-square in $\bF_p$.  If $p\equiv 1$ (mod $4$) and $ab$ is a square in $\bF_p$, the 6 systems are all consistent and have 24 solutions and by direct computation, we can conclude that they are the preimages of $\Phi_{2 \times 4}(a,b)$. 
Therefore,
\begin{align*}
|W_{ 2 \times 4, J}|  = \left\{
\begin{array}{cl}
24 & \textrm{if }  \sqrt{ab} \in \bF_p  \textrm{ and } p \equiv 1 \pmod{4},\\
8 & \textrm{otherwise} . \\
\end{array}
\right.
\end{align*}

We recall that $E_{J}$ has three non-trivial 2-torsion points whose $x$-coordinates are $-(3a^2 - 18ab + 3b^2),-(3a^2 +18ab + 3b^2),  (6a^2 + 6b^2)$ respectively. The points $P$ with $2P=(6a^2+6b^2,0)$  are already included in the $E_J(\bF_p)$, and one can check that the two points
\begin{align*}
(3a^2 + 18ab + 3b^2 \pm 18\sqrt{ab}(a+b),  \sqrt{-1} \cdot 2 \cdot 3^3 \sqrt{ab}(\sqrt{a} \pm \sqrt{b})^2(a+b))
\end{align*}
and their inverses defined in $\bF_p[\sqrt{ab}, \sqrt{-1}]$ are the 4 points $Q$ with $2Q=(-(3a^2 +18ab + 3b^2), 0)$. 
Therefore, $p \equiv 1 \pmod{4}$ and $ab$ is square if and only if $E_J$ includes $\bZ/4\bZ \times \bZ/4\bZ$ which is equivalent to $|W_{2 \times 4, J}| = 24$.

At last, we need to show that when $E_{J}$ has $\bZ/2\bZ \times \bZ/4\bZ$ as a subgroup, then there exists $a, b \in \bF_p$ such that $J = \Phi_{2 \times 4}(a, b)$.
Since we already showed the analogue for $G = \bZ/2\bZ \times \bZ/2\bZ$ and $\bZ/4\bZ$, there are $u, v, s, t \in \bF_p$ such that
\begin{align*}
	(A, B) = (-(u^2+3v^2)/4, (v^3 - u^2v)/4) = (-3s^2 + 6st^2 -2t^4, (2s - t^2)(s^2 + 2st^2 - t^4)).
\end{align*} 
One can check that $5t^2-12s$ should be square, say $r^2$. Then, $(A, B) = \Phi_{2 \times 4}(6^{-1}r, 6^{-1}t)$.
\end{proof}

Hence for example, 
\begin{equation*}
|W_{6, J}|  = \left\{
\begin{array}{cl}
24 & \textrm{if } E_J(\bF_p)[6] \cong \bZ/6\bZ \times \bZ/6\bZ, \\
8 & \textrm{if } E_J(\bF_p)[6] \cong \bZ/6\bZ \times \bZ/3\bZ, \\
6 & \textrm{if } E_J(\bF_p)[6] \cong \bZ/6\bZ \times \bZ/2\bZ, \\
2 & \textrm{if }  E_J(\bF_p)[6] \cong \bZ/6\bZ, \\
0 & \textrm{if }  E_J(\bF_p)[6] \not\geq \bZ/6\bZ.
\end{array}
\right.
\end{equation*}

Comparing to Lemma \ref{assumption}, 
we should remark that the analogues result does not hold for all torsion groups.
\begin{example}
Let $E_1 : y^2 = x^3 + 2x+1$ and $E_2 : y^2 = x^3 + 2x + 4$ be elliptic curves over $\bF_5$.
Then, $E_1$ and $E_2$ are isomorphic, and $E_1(\bF_5)\cong E_2(\bF_5) \cong \bZ/7\bZ$.
However, one can compute that $|W_{7, (2, 1)}| = 0$ and $|W_{7, (2, 4)}| = 12$.
\end{example}


\subsection{Moments of traces of the Frobenious} \label{subsec:moments of class numbers}
Now, we define a class number weighted by $|W_{G, J}|$.
\begin{definition}
We define
\begin{equation} \label{eqn:defCap}
H_G(a,p) := \sum_{\substack{ J = (A, B) \in (\bZ/p\bZ)^2 \\  a_p(E_{J }) = a\\ 4 A^3 + 27B^2 \not\equiv 0 \pmod{p}} } |W_{G, J}|,
\end{equation}
where $a_p(E)$ is a trace of the Frobenius of an elliptic curve $E$ at $p$.
\end{definition}

The goal of this section is to show
\begin{align} \label{moment_torsion 1}
\sum_{|a|< 2\sqrt{p}}H_G(a,p)&=p^2+O_G(p),\\
\sum_{|a|< 2\sqrt{p}}aH_G(a,p) &=O_G(p^{\frac{3}{2}}), \label{moment_torsion 2} \\
\sum_{|a|< 2\sqrt{p}}a^2 H_G(a,p)& =p^3+O_G(p^{\frac{5}{2}}). \label{moment_torsion 3}
\end{align}

The main tool is the Eichler--Selberg trace formula \cite{KP2}.
We recall some notations first.
The Chebyshev polynomials of the second kind are defined as
\begin{align*}
U_0(t)=1, \quad U_1(t)=2t, \quad U_{j+1}(t)=2tU_j(t)-U_{j-1}(t).
\end{align*}
We define normalized Chebyshev polynomials to be
\begin{align*}
U_{k-2}(t,q):=q^{k/2-1}U_{k-2}\left( \frac{t}{2\sqrt{q}}\right)=\frac{\alpha^{k-1}-\overline{\alpha}^{k-1}}{\alpha-\overline{\alpha}} \in \mathbb{Z}[q,t],
\end{align*}
where $\alpha, \overline{\alpha}$ are the two roots in $\mathbb{C}$ of $X^2-tX+q=0$. Let
\begin{align*}
C_{R,j} :=\begin{cases}
a_{\frac{R}{2},j} & \text{ if $R$ is even} \\
a_{\frac{R-1}{2},j} +a_{\frac{R-1}{2},j-1}  & \text{ if $R$ is odd} 
\end{cases}
\qquad \textrm{for } a_{R,j}:={2R \choose j} -{2R \choose j-1}
\end{align*}
be the Chebyshev coefficients. We have
$$
t^R=\sum_{j=0}^{\lfloor R/2 \rfloor}C_{R,j}q^j U_{R-2j}(t,q)
$$
which is \cite[(1.3)]{KP2}. In particular we have
\begin{align} \label{moment}
t^0=U_0(t,q),\quad t=U_1(t,q), \quad t^2= U_2(t,q) + qU_0(t,q).
\end{align}

Let $E$ be an elliptic curve defined over a finite field $\mathbb{F}_q$ with $q$ elements, $\frak{C}$ be the set of all the isomorphism classes of elliptic curves over  $\mathbb{F}_q$ . 
Let $A$ denote a finite abelian group and let $\Phi_A$ to be 
\begin{align*}
\Phi_A(E)=\begin{cases}
1  & \text{ if there exists an injective homomorphism $A \hookrightarrow E(\mathbb{F}_p)$ }\\
0 & \text{otherwise}.
\end{cases}
\end{align*}
We define
\begin{align*}
\mathbb{E}_q(a^R\Phi_A) :=\frac{1}{q} \sum_{\substack{E \in \frak{C} \\ A \hookrightarrow E(\mathbb{F}_q ) }} \frac{a_q(E)^R}{|\text{Aut}_{\mathbb{F}_q}(E)|}.
\end{align*}


From now on, we assume that $q=p$. 
For a finite abelian group $A$, let $n_1=n_1(A)$ and $n_2=n_2(A)$ be its first and second invariant factors, respectively.
Also, we denote $\psi(n) = n \prod_{p \mid n} (1 + 1/p)$, $\varphi(n) = n\prod_{p \mid n} (1 - 1/p)$ and $\phi(n) = n \prod_{p \mid n}( - \varphi(p))$.

For $\lambda \mid (p-1, n_1)$, let
\begin{align*}
T_{n_1,\lambda}(p,1) := \frac{\psi(n_1^2/\lambda^2) \varphi(n_1/\lambda)}{\psi(n_1^2)}(-T_{\trace} -T_{\hyp} + T_{\dual}),
\end{align*}
with 
\begin{align*}
T_{\trace}& :=\frac{1}{\varphi(n_1)} \Tr(T_p|S_k(\Gamma(n_1,\lambda))),\\
T_{\hyp}& : =\frac{1}{4}\sum_{i=0}^1\sum_{\substack{\tau |n_1 \lambda \\ g |p-1}}\frac{\varphi(g)\varphi(n_1(n_1(\lambda,g)/g)}{\varphi(n_1)}\left(\delta_{n_1(\lambda,g)/g}(y_i,1)+(-1)^k \delta_{n_1(\lambda,g)/g}(y_i,-1)\right),\\
T_{\dual}&:=\frac{p+1}{\varphi(n_1)}\delta(k,2),
\end{align*}
where $g=(\tau, n_1\lambda/\tau)$, $y_i$ is the unique element of $(\bZ/(n_1\lambda/g)\bZ)^\times$ such that $y_i \equiv p^i$ (mod $\tau$) and $y_i \equiv p^{1-i}$ (mod $n_1\lambda /\tau)$,
$\delta(a, b)$ is the indicator function of $a = b$, and $\delta_c(a,b)$ is the indicator function of the congruence $a \equiv b \pmod{c}$.

\begin{theorem}\cite[Theorem 3, when $q=p$]{KP2} \label{thm:KP}
Let $A$ be a finite abelian group of rank at most 2. Suppose $(p,|A|)=1$ and $k\geq 2$. If $p \equiv 1$ (mod $n_2(A)$) we have
\begin{align}
\mathbb{E}_p(U_{k-2}(t,p)\Phi_A)=\frac{1}{\varphi(n_1/n_2)}\sum_{\nu | \frac{(p-1,n_1)}{n_2} } \phi(\nu) T_{n_1, n_2\nu}(p,1)
\end{align}
and if $p \not\equiv 1$ (mod $n_2(A)$), then $\mathbb{E}_p(U_{k-2}(t,p)\Phi_A)=0$. 
\end{theorem}

\begin{proposition}
Let $G$ be one of the groups $\bZ/n\bZ$ for $2 \leq n \leq 6$ or $\bZ/2\bZ \times \bZ/2\bZ$.
Then, (\ref{moment_torsion 1}), (\ref{moment_torsion 2}) and (\ref{moment_torsion 3}) hold.
\end{proposition}
\begin{proof}
For each group $G$, we denote $n_1$ be its first invariant factor.
We define $A_{G, i}$ be abelian groups satisfying $G \leq A_{G, i} \leq \bZ/n_1\bZ \times \bZ/n_1\bZ$,
and $j < i$ if and only if $A_{G, j } < A_{G, i}$.
We define $\widetilde{\omega}_{G, i}$ to be  $\left|W_{G, I} \right|$ if $E_I[n_1](\bF_p) \cong A_{G, i}$. This is well defined by Lemma \ref{prop:value of WGI}.
Let
\begin{equation*}
\omega_{G,i} := \widetilde{\omega}_{G, i} - \sum_{j < i} \omega_{G,j}.
\end{equation*}
Then, one can obtain that 
\begin{align*}
\sum_{|a| <2 \sqrt{p}}a^R H_G(a,p) 
= p(p-1)\sum_i \omega_{G,i} \mathbb{E}_p(a^R\Phi_{A_{G,i}})
\end{align*}
For arbitrary $G$, we can show that 
\begin{equation*}
	\sum_{|a| < 2 \sqrt{p}}H_G(a,p)=p^2+O(p),
\end{equation*}
by Proposition \ref{prop:sum of WGI}.
Hence, 
\begin{equation} \label{eqn:U0_sum}
	\sum_i \omega_{G,i} \mathbb{E}_p(\Phi_{A_{G,i}}) = 1 + O\lbrb{\frac{1}{p}}.
\end{equation}
Since $t^2= U_2(t,p) + pU_0(t,p)$, we have the identity
\begin{align*}
\mathbb{E}_p(t^2\Phi_{A}) = \mathbb{E}_p(U_2(t,p)\Phi_{A}) + p\mathbb{E}_p(U_0(t,p)\Phi_{A}) 
\end{align*}
and this together with (\ref{eqn:U0_sum}) implies
\begin{align*}
\sum_{|a| <2 \sqrt{p}}a^2H_G(a,p)=p(p-1)(p+O(1)) + O(p^{2.5})=p^3+O(p^{2.5})
\end{align*}
because $\mathbb{E}_p(U_2(t,p)\Phi_{A}) \ll_G \frac{p^{1.5}}{p}\ll_G p^{0.5}$
by Theorem \ref{thm:KP} and Deligne bound.

Using the identity $t=U_1(t,p)$ and $\mathbb{E}_p(U_1(t,p)\Phi_{A}) \ll_G p^{-0.5}$, it is easy to see that
$$
\sum_{|a| <2 \sqrt{p}}aH_G(a,p)=O_G(p^{1.5}),
$$
by Theorem \ref{thm:KP} and Deligne bound.
\end{proof}
When $G = \bZ/2\bZ$ or $\bZ/2\bZ \times \bZ/2\bZ$, we can obtain the $2R+1$-th moments.

\begin{proposition}
When $G = \bZ/2\bZ$ or $\bZ/2\bZ \times \bZ/2\bZ$, we have
\begin{equation*}
	\sum_{|a| < 2\sqrt{p}}a^{2R + 1}H_G(a, p) = 0
\end{equation*}
for $R \geq 0$.
\end{proposition}
\begin{proof}
	Let $N_n(a)$ (resp. $N_{n \times n}(a)$) be the number of isomorphism classes of elliptic curves over $\bF_p$ such that $E(\bF_p)[n] \geq \bZ/n\bZ$ (resp. $E(\bF_p)[n] = \bZ/n\bZ \times \bZ/n\bZ$) with weights $2/|\Aut_{\bF_p}(E)|$. Then, \cite[Theorem 4.6, 4.9]{Sch} shows that for a prime $p \geq 5$, an $a$ in the Weil bound, and a positive integer $n \geq 2$,
	\begin{equation*}
N_n(a) = \left\{ \begin{array}{ccc}
H(a^2 - 4p) & \textrm{if $a \equiv p+1 \pmod{n}$,} \\
0 & \textrm{otherwise,}
\end{array} \right.
\end{equation*}
and
\begin{equation*}
N_{n\times n}(a) = \left\{ \begin{array}{ccc}
H\lbrb{\frac{a^2 - 4p}{n^2}} & \textrm{if $p \equiv 1 \pmod{n}$ and $a \equiv p+1 \pmod{n^2}$,} \\
0 & \textrm{otherwise.}
\end{array} \right.
\end{equation*}
By Lemma \ref{prop:value of WGI},
\begin{equation*}
	H_2(a, p) = \frac{p-1}{2}(N_2(a) + 2N_{2\times 2}(a)), \qquad
	H_{2 \times 2}(a, p) = 6 \cdot \frac{p-1}{2} N_{2\times 2}(a).
\end{equation*}
Since $N_2(a) = N_2(-a)$ and $N_{2 \times 2}(a) = N_{2 \times 2}(-a)$, the result follows.
\end{proof}

This will be used for the Frobenius trace formula for elliptic curves.

\section{Counting elliptic curves with torsion points and local conditions} \label{sec:main}
We introduce some notations first. Let
\begin{equation*}
	R_G(X) = \lcrc{(a,b) \in \bR^2 : | f_G(a,b) | \leq X^{\frac{1}{3}}, |g_G(a,b)| \leq X^{\frac{1}{2}} }.
\end{equation*}

For $G$ in $\mathcal{G}_{\leq 4}$ we define
\begin{align*}
	\cD_G(X) & = \lcrc{(A,B) \in \bZ^2 : (A,B) = \Phi_G(a,b) \textrm{ for some } (a, b) \in R_G(X) \cap \bZ^2}, \\
	\cM_G(X) & = \lcrc{(A,B) \in \cD_G(X) : \textrm{ if $p^4 \mid A$, then  $p^6 \nmid B$} },
\end{align*}
and
\begin{equation*}
	\cE_G(X) = \lcrc{(A, B) \in \cM_G(X) : 4A^3 + 27B^2 \neq 0}, \qquad
	\cS_G(X) = \lcrc{(A, B) \in \cM_G(X) : 4A^3 + 27B^2 = 0},
\end{equation*}
where $\cE_G(X)$ represents elliptic curves with $G$ torsion and $\cS_G(X)$ takes up singular curves. 
We note that $\cE_G(X)$ coincide with the previous definition.

For $G \in \cG_{\geq 5}$, we recall that $M_G(X)$ is the set of relatively prime pairs  $(a,b)$ with $h(\Phi_G(a,b))\leq X$. 
We define 
\begin{align*}
	\widetilde{M}_G^e(X) = \lcrc{(a, b) \in \bZ^2 : (a, b) = 1, e = e(a, b), |f_G(a, b)| \leq X^{\frac{1}{3}}, |g_G(a, b)| \leq X^{\frac{1}{2}}},
\end{align*}
and $\widetilde{M}_G(X)$ as the union of $\widetilde{M}_G^e(X)$ for all $e \geq 1$. We define $\cE_G(X)$ as (\ref{eqn:EG large}) and 
\begin{align*}
	\cS_G(X) = \lcrc{(A, B) \in \cS(X) : \Phi_G(a, b) \textrm{ for relatively prime } (a, b)}
\end{align*}
where 
\begin{equation*}
\cS(X) = \lcrc{(A, B) \in \bZ^2 : 
\begin{array}{cc}
|A| \leq X^{\frac{1}{3}}, |B| \leq X^{\frac{1}{2}}, 4A^3 + 27B^2 = 0, \\ 
\textrm{if $p^4$ divides $A$, then $p^6$ does not divide $B$. }
\end{array}
}.
\end{equation*}

For the reader's convenience we remark that $(a, b)$ denotes an element in the domain of $\Phi_G$ and $R_G$ (resp.  $M_G$  for $G$ in $\mathcal{G}_{\geq 5}$) and $(A,B)$ does in the range of $\Phi_G$. Also, $\cD_G, \cM_G, \cE_G,$ and $\cS_G$ are sets on the range side. 
For pairs $I, J \in (\bZ/p\bZ)^2$, the subscripts $-_{G, I}(X)$ or $-_{G, J}(X)$ means that this is the subset
of the original set  consisting of elements $(a,b)\equiv I$ $\pmod{p}$ or $(A,B)\equiv J$ $\pmod{p}$ respectively. We often drop the subscript $G$ to ease the notation.


\begin{lemma} \label{lem:pointsinRGX}
For a torsion subgroup $G$,  the number of integer points in $R_G(X)$ is
\begin{equation*}
\Area(R_G(1)) X^{\frac{1}{d(G)}} + O(X^{\frac{1}{e(G)}}).
\end{equation*}
\end{lemma}
\begin{proof}
We note that \cite[Lemma 5.2]{HS} proves this lemma for $G = \bZ/2\bZ, \bZ/3\bZ$.
Since $f_4(a,b) = X^{\frac{1}{3}}, g_4(a,b) = X^{\frac{1}{2}}$ are equivalent to 
$f_4(a/X^{\frac{1}{6}}, b/X^{\frac{1}{12}}) = 1, g_4(a/X^{\frac{1}{6}}, b/X^{\frac{1}{12}})=1$, by change of variables we have
\begin{equation*}
\textrm{Area}(R_4(X)) = X^{\frac{1}{4}} \textrm{Area}(R_4(1)).
\end{equation*}
Then, the claim follows from the Principle of Lipschitz, \cite[(5.3)]{HS}.
We can do the same thing for $G = \bZ/2\bZ\times \bZ/2\bZ$.
Also, we obtain the result for the groups $G$ in $\mathcal{G}_{\geq 5}$ since $3 \deg f_G(a,b) = 2 \deg g_G(a, b) = 2d(G)$.
\end{proof}

By the Principle of Lipschitz, we have
\begin{corollary} \label{Dtildeest}
For a prime $p \geq 5$, $I$ an element in $(\bZ/p\bZ)^2$, and a torsion subgroup $G$, we have
\begin{equation*} 
|R_{G, I}(X)| =
\Area(R_G(1))  p^{-2}X^{\frac{1}{d(G)}} + O( 1+ p^{-1}X^{\frac{1}{e(G)} }).
\end{equation*} 
\end{corollary}

For $G=\bZ/2\bZ \times \bZ/2\bZ$, we consider only the pairs $(a,b)$ with $a\equiv b$ $\pmod{2}$. 
By Lemma \ref{lem:pointsinRGX} and M\"obious inversion argument gives the following corollary, which is a complement of \cite[Theorem 5.6]{HS}. For details, we refer to the proof of Proposition  \ref{prop:tor cE}.
\begin{corollary} \label{cor:HS4 22case}
For $G$ in $\mathcal{G}_{\leq 4}$, let
\begin{align*}
	c(G) := \frac{\Area(R_G(1))}{2^{\delta_{G=2\times 2}}r(G)\zeta(\frac{12}{d(G)})}.
\end{align*}
Then,
\begin{align*}
	|\cE_G(X)| = c(G)X^{\frac{1}{d(G)}} + O(X^{\frac{1}{e(G)}}).
\end{align*}
\end{corollary}

\begin{lemma} \label{DDtilde}
For a prime $p \geq 5$, a non-zero $J$  in $(\bZ/p\bZ)^2$ and a group G in $\mathcal{G}_{\leq 4}$,
\begin{equation*}
|\cD_{G, J}(X)| 
 = \frac{ | W_{G, J} | }{2^{\delta_{G=2\times 2}}r(G)}|R_{G, I}(X)| + O(1+ p^{-1}X^{\frac{1}{ e(G)   }} )
\end{equation*}
where $I \in W_{G,J}$.
For $G \in \cG_{\geq 5}$, we have
\begin{align*}
	|\cE_{G, J}(X)| = \frac{|W_{G, J}|}{r(G)} \sum_{I \in W_{G,J}}\sum_{e} |M_{I}^e(X)| + O(1 + p^{-1}X^{\frac{1}{e(G)}}).
\end{align*}
\end{lemma}
\begin{proof}
We fix a group $G$ and omit it from subscription. 
 For $G \in \cG_{\leq 4}$, $\Phi$ induces a surjective map 
 \begin{align*}
 \bigsqcup_{I \in W_J} R_{I}(X) \to  \cD_{J}(X), \qquad (a,b) \rightarrow (A,B)= \Phi(a,b).
 \end{align*}
 Let $h_{J}(X)$ be the number of the $2$-tuples $(A,B)$ for which its pre-image is not equal to $r(G)$. Then, $h_J(X)$ is bounded by $O(p^{-1}X^{\frac{1}{e(G)}})$ by the proof of Lemma \ref{lem:preimage rG}.
We note that  the number of solutions of a system of equations
 \begin{equation*}
 f(a,b) =A \qquad \textrm{and } \qquad g(a,b) = B
 \end{equation*}
 is less than or equal to $\deg f \cdot \deg g$ by Bezout's theorem and $|R_I(X)|$ does not depend on $I$ by Corollary \ref{Dtildeest}. 
 Therefore, we have
 \begin{equation*}
 |\cD_{J}(X)|   = \frac{|W_{J}|}{2^{\delta_{G=2\times 2}}r(G)}|R_{I}(X)| + O\lbrb{h_{J}(X) }.
 \end{equation*}

For $G \in \cG_{\geq 5}$, $\Phi$ induces a surjective map
\begin{align*}
	\bigcup_{I \in W_J}\bigcup_{e}M_{I}^{e}(X) \to \cE_{J}(X) \bigcup \cS_J(X).
\end{align*}
Hence, the above argument and an estimate of $\cS_J(X)$ give a similar result.
\end{proof}

For a pair $(A,B)$ of integers or elements of $\bZ/p\bZ$ and an integer $d$, we define an operation $*$ by $d*(A,B) = (d^4A, d^6B)$.

\begin{proposition} \label{cMcD}
For a prime $p \geq 5$, non-zero $J \in (\bZ/p\bZ)^2$, and a group $G$ in $\mathcal{G}_{\leq 4}$,
\begin{equation*}
|\cM_{G, J}(X)| = \sum_{\substack{d \leq X^{\frac{1}{12}} \\ p \nmid d}}
\mu(d) |\cD_{G,  d^{-1}*J}(d^{-12}X)|,
\end{equation*}
and $|\cS_{G, J}(X)| = O(X^{\frac{1}{6}}/p)$.
\end{proposition}
\begin{proof}
Let $(A, B) \in \cD_{G, J}(X)$ and let $d$ be the maximum of $d'$ satisfying $d'^4 \mid A$ and $d'^6 \mid B$. 
Since $J$ is non-zero, $p \nmid d$.
By (\ref{eqn:def fGgG}), the definition of $f_G$ and $g_G$, one can easily check that there are positive integers $m$ and $n$ depending $G$ such that 
\begin{equation*}
\frac{1}{d^4} f_G(a,b) = f_G\left(\frac{a}{d^{m}}, \frac{b}{d^{n}} \right), \qquad
\frac{1}{d^6} g_G(a,b) = g_G\left(\frac{a}{d^{m}}, \frac{b}{d^{n}} \right),
\end{equation*}
for given $G$. 
Also, we can check that  $a/d^{m}$ and $b/d^{n}$ are integers.
Hence if $(A, B) = (f_G(a,b), g_G(a,b))$ for some  $a, b\in\bZ$, then  
\begin{equation*}
d^{-1}*(A, B) = \left(f_G \left(\frac{a}{d^{m}}, \frac{b}{d^{n}} \right), g_G\left(\frac{a}{d^{m}}, \frac{b}{d^{n}}\right)\right)
\end{equation*}
is an element of $\cM_{G, d^{-1}*J}(d^{-12}X)$ and 
there is a bijection
$$(A,B) \to d^{-1}*(A,B), \qquad \cD_{G, J}(X) \to \bigsqcup_{\substack{d \leq X^{\frac{1}{12}} \\ p \nmid d}} \cM_{G, d^{-1}*J}(d^{-12}X).$$
By M\"obius inversion argument, the first equality follows.
The error term is easy to establish.
\end{proof}

\begin{proposition} \label{prop:tor cE}
For a non-zero $2$-tuple $J \in (\bZ/p\bZ)^2$ where $p \geq 5$, $G$ in $\mathcal{G}_{\leq 4}$  we have
\begin{align*}
|\cE_{G, J}(X)| = c(G)\frac{|W_{G,J}| }{p^{2}}\frac{p^{\frac{12}{d(G)}}}{p^{\frac{12}{d(G)}}-1} X^{\frac{1}{d(G)}} +
O(p^{-1}X^{\frac{1}{e(G)}}   + X^{\frac{1}{12}}). 
\end{align*}
For $J = (0,0)$, we have
\begin{align*}
|\cE_{G, J}(X)| = c(G) \lbrb{ \frac{1}{p^2} - \frac{1}{p^{\frac{12}{d(G)}}}} 
\frac{p^{\frac{12}{d(G)}}}{p^{\frac{12}{d(G)}}-1} X^{\frac{1}{d(G)}} + O(pX^{\frac{1}{e(G)}}   + p^2 X^{\frac{1}{12} }). 
\end{align*}

\end{proposition}

\begin{proof}
For $d$ not divisible by $p$,
we note that $|W_{G, d^{-1}*J}| = |W_{G, J}|$ for all $p \nmid d$. Then, 
\begin{align*}
|\cE_{G, J}(X)|&=  |\cM_{G,  J}(X)| +O(\cS_{G, J}(X) )
= \sum_{\substack{d \leq X^{\frac{1}{12}} \\ p \nmid d}} \mu(d)|\cD_{G, d^{-1}*J}(\frac{X}{d^{12}})| 
+ O\lbrb{ \frac{X^{\frac{1}{6}}}{p} }
\\
&= \sum_{\substack{d \leq X^{\frac{1}{12}} \\ p \nmid d}} \mu(d) 
\left( \frac{|W_{G, d^{-1}*J} | }{2^{\delta_{G=2\times 2}}r(G)}  | R_{G, I}(\frac{X}{d^{12}})|
+ O\lbrb{ 1+ p^{-1 } \frac{X^\frac{1}{e(G)}}{d^{\frac{12}{e(G)}}} }\right) + O\lbrb{ \frac{X^{\frac{1}{6}}}{p} } \\
&= \frac{ | W_{G,J} |}{2^{\delta_{G=2\times 2}}r(G)} \sum_{\substack{d \leq X^{\frac{1}{12}} \\ p \nmid d}} \mu(d)  |R_{G, I}(\frac{X}{d^{12}})|
+ O\lbrb{ X^{\frac{1}{12}} + \frac{X^{\frac{1}{e(G)}}}{p} },
\end{align*}
by Lemma \ref{DDtilde}.
Here we also used that $|R_{G, I}(X)|$ does not depend on $I \in W_J$.
Using Corollary \ref{Dtildeest}, the sum is
\begin{align*}
&= \frac{ |W_{G,J} | \Area(R_G(1))}{2^{\delta_{G=2\times 2}}r(G) p^{2}} \sum_{\substack{d \leq 
X^{\frac{1}{12}} \\ p \nmid d}} \mu(d) \lbrb{  \frac{X^{\frac{1}{d(G)}} }{ d^{\frac{12}{d(G)}  } }+ O\left(\frac{p X^{\frac{1}{e(G)}}}{  d^{\frac{12}{e(G)}} } +p^2 \right)  } 
  + O\lbrb{ X^{\frac{1}{12}} +  \frac{X^{\frac{1}{e(G)}}}{p} } \\
 &= c(G) \frac{| W_{G, J} |}{p^{2}  }\frac{p^{\frac{12}{d(G)}}}{p^{\frac{12}{d(G)}}-1} X^{\frac{1}{d(G)}} + O\left( \frac{X^{\frac{1}{e(G)}  }}{p} +  X^{\frac{1}{12}}  \right) .
\end{align*}

By \cite[Theorem 5.6]{HS} and Corollary \ref{cor:HS4 22case}, the main term of 
\begin{align*}
|\cE_{G, (0,0)}(X)| &= |\cE_G(X)| - \sum_{J \neq (0,0)} |\cE_{G,  J}(X)| 
\end{align*}
is
\begin{align*}
c(G)X^{\frac{1}{d(G)}} -\sum_{J \neq (0,0)}c(G) \frac{ | W_{G,J} | }{p^{2}}\frac{p^{\frac{12}{d(G)}}}{p^{\frac{12}{d(G)}}-1} X^{\frac{1}{d(G)}}  &=
c(G)X^{\frac{1}{d(G)}} \lbrb{\frac{p^{\frac{12}{d(G)}}-1}{p^{\frac{12}{d(G)}}} -  \frac{\sum_{ J \neq (0,0)}| W_{G,J}| }{p^{2}} }
\frac{p^{\frac{12}{d(G)}}}{p^{\frac{12}{d(G)}}-1} \\
&= c(G)X^{\frac{1}{d(G)}} \lbrb{ \frac{1}{p^2} - \frac{1}{p^{\frac{12}{d(G)}}}}
\frac{p^{\frac{12}{d(G)}}}{p^{\frac{12}{d(G)}}-1}.
\end{align*}
This gives the main term, and the error term is easily checked. 
\end{proof}

 For torsion $G$, we define
\begin{align*}
	c_{G, \cLC}(p) = \sum_{E_I \textrm{ satisfies } \cLC} \frac{|W_{G, I}|}{p^2}.
\end{align*}

\begin{theorem} \label{2tor bad}
For a prime $p \geq 5$, a local condition $\cLC$, and a group $G $ in $ \mathcal{G}_{\leq 4}$, 
\begin{equation*}
|\cE_{G, p}^{\cLC}(X)| = c(G) \cdot  c_{G,\cLC}(p) \cdot \frac{p^{\frac{12}{d(G)}}}{p^{\frac{12}{d(G)}}-1} X^{\frac{1}{d(G)}} + O(h_{G,\cLC}(p, X))
\end{equation*}
where $c_{G,\cLC}(p)$ is 
\begin{equation*}
\begin{array}{|c|c|c|c|c|}
\hline
	& 2 & 3 & 4 &  2\times 2  \\
	\hline
\mathrm{good}	& (p-1)^2/p^2 & (p-1)^2/p^2 & (p-1)(p-2)/p^2 & (p-1)(p-2)/p^2 \\
	\hline
\mathrm{mult}	& (2p - 2)/p^2 & (2p - 2)/p^2 & (3p - 3)/p^2 & (3p - 3)/p^2 \\
	\hline
\mathrm{addi}	&  1/p^2 - 1/p^{6}   &  1/p^2 - 1/p^{4}  &  1/p^2 - 1/p^{3}  &  1/p^2 - 1/p^{4}  \\
\hline
a & H_2(a,p)/p^2 & H_3(a,p)/p^2 & H_4(a,p)/p^2 & H_{2\times 2}(a,p)/p^2 \\
\hline
\end{array}
\end{equation*}
and
\begin{equation*}
c_{3, \mathrm{split}}(p) =  
\begin{cases}  2(p-1)/p^2 &  \mbox{ for $p\equiv 1$ mod $12$, }  \\ 
(p-1)/p^2 &  \mbox{ for $p\equiv 5$ or $11$ mod $12$, } \\
0 &  \mbox{ for $p\equiv 7$ mod $12$.}
\end{cases}
\end{equation*}
Finally for $\epsilon > 0$, the function $h_{G,\cLC}(p,X)$ is 
\begin{equation*}
\begin{array}{|c|c|}
\hline
	& h_{G,\cLC}(p,X)\\
	\hline
\mathrm{good/bad}	& pX^{\frac{1}{e(G)}} + p^2X^{\frac{1}{12} }  \\
	\hline
\mathrm{mult}	& X^{\frac{1}{e(G)}} + pX^{\frac{1}{12}}  \\
	\hline
\mathrm{split}	& X^{\frac{1}{e(G)}} + pX^{\frac{1}{12} }  \\
	\hline
\mathrm{addi}	& pX^{\frac{1}{e(G)}} + p^2X^{\frac{1}{12} }  \\
\hline
a & H_G(a,p)(p^{-1}X^{\frac{1}{e(G)}} + X^{\frac{1}{12}} ) \\
\hline
\end{array}
\end{equation*}

\end{theorem}
\begin{proof}
By Proposition \ref{prop:tor cE},
\begin{align*}
|\cE_{G, p}^{\textrm{good}}(X)| &= \sum_{\substack{ J = (A, B) \in \bF_p^2 \\ 4A^3 + 27 B^2 \not\equiv 0}} |\cE_{G, J}(X)|
\\
&= \sum_{\substack{ J = (A, B) \in \bF_p^2 \\ 4A^3 + 27 B^2 \not\equiv 0}}c(G)\frac{| W_{G,J} | }{p^2}\frac{p^{\frac{12}{d(G)}}}{p^{\frac{12}{d(G)}}-1} X^{\frac{1}{d(G)}} +  O\left( p(p-1) \left(\frac{X^{\frac{1}{e(G)}}}{p} + X^{\frac{1}{12} } \right)  \right).
\end{align*}
By Propositions \ref{prop:sum of WGI}, we have
\begin{align*}
&\sum_{\substack{ J = (A, B) \in \bF_p^2 \\ 4A^3 + 27 B^2 \not\equiv 0}} | W_{2, J} | =\sum_{\substack{ J = (A, B) \in \bF_p^2 \\ 4A^3 + 27 B^2 \not\equiv 0}}|W_{3, J} | = (p-1)^2, \\
&\sum_{\substack{ J = (A, B) \in \bF_p^2 \\ 4A^3 + 27 B^2 \not\equiv 0}} | W_{4, J} | = 
\sum_{\substack{ J = (A, B) \in \bF_p^2 \\ 4A^3 + 27 B^2 \not\equiv 0}} | W_{2\times 2, J} | = (p-1)(p-2).
\end{align*}
This proves good reduction cases. 
The other cases can be shown similarly.
\end{proof}

For $G \in \cG_{\geq 5}$,
we note that 
\begin{align*}
	|\widetilde{M}_G(X)| = \frac{\Area(R(1))}{\zeta(2)}X^{\frac{1}{d(G)}} + O(X^{\frac{1}{2d(G)}} \log X)
\end{align*}
by the M\"obius inversion and the Principle of Lipschitz. 
For details, we refer to the proof of the following lemma.

\begin{lemma} \label{lem:MI for Ggeq}
For arbitrary prime power $p^m$ and a pair $I \in (\bZ/p^m\bZ)^2$ whose coordinates are not divided by $p$ simultaneously,
\begin{align*}
	|\widetilde{M}_{G, I}(X)| = \frac{1}{p^{2m}} \frac{p^2}{p^2-1} \frac{\Area(R(1))}{\zeta(2)}X^{\frac{1}{d(G)}} + O(X^{\frac{1}{2d(G)}} +p^{-m}X^{\frac{1}{2d(G)}} \log X).
\end{align*}
\end{lemma}
\begin{proof}
For a given $I$, we have a bijection
\begin{align*}
	R_{G, I}(X) \cong \bigsqcup_{\substack{d \leq X^{\frac{1}{2d(G)}} \\ p \nmid d}} d * \widetilde{M}_{G, d^{-1}*I}(d^{-2d(G)}X), \qquad
	(a, b) \to d*\lbrb{\frac{a}{d}, \frac{b}{d}},
\end{align*}
where $d$ is the gcd of $a$ and $b$. By M\"obius inversion argument and Corollary \ref{Dtildeest}, we have 
\begin{align*}
	|\widetilde{M}_{G, I}(X)| &= \sum_{\substack{d \leq X^{\frac{1}{2d(G)}} \\ p \nmid d}} \mu(d) |R_{G, d^{-1}*I}(d^{-2d(G)}X)| \\
	&= \sum_{\substack{d \leq X^{\frac{1}{2d(G)}} \\ p \nmid d}} \mu(d)
	\lbrb{ \frac{1}{p^{2m}} \Area(R(1))  \frac{ X^{\frac{1}{d(G)}} }{d^2} + O(p^{-m}d^{-1}X^{\frac{1}{2d(G)}})} \\
	& = \frac{1}{p^{2m}} \frac{p^2}{p^2-1} \frac{\Area(R(1))}{\zeta(2)}X^{\frac{1}{d(G)}} + O(X^{\frac{1}{2d(G)}} + p^{-m}X^{\frac{1}{2d(G)}} \log X).
\end{align*}
\end{proof}

\begin{theorem} \label{local_count}
	Let $G$ be a torsion subgroup in $\cG_{\geq 5}$, $p \geq 5$ be a prime, and $J$ be a  non-zero element of $(\bZ/p\bZ)^2$.
Then, there is an absolute constant $c(G)$ such that
\begin{equation*}
	|\cE_{G, J}(X)| = \frac{|W_{G,J}|}{p^2-1} c(G) X^{\frac{1}{d(G)}} + O(X^{\frac{1}{2d(G)}} + p^{-1}X^{\frac{1}{2d(G)}} \log X).
\end{equation*}
\end{theorem}
\begin{proof}
We use the strategy of \cite[\S 3]{CKV}.
Let $\epsilon = \epsilon(G)$ be a positive integer which is the least common multiplier of the possible defects of $(f_G, g_G)$ which is well-defined by Lemma \ref{lem:defect}.
Since $M_G^e(X) = \widetilde{M}_G^e(e^{12}X)$,
\begin{align*}
	|\cE_{G, J}(X)| &= \frac{1}{r(G)}\sum_{I \in W_J}\sum_{e \mid \epsilon}|M^e_{G, I}(X)| + O(1+p^{-1}X^\frac{1}{e(G)}) = \frac{1}{r(G)}\sum_{I \in W_{G,J}} \sum_{e \mid \epsilon}|\widetilde{M}^e_{G, I}(e^{12}X)| + O(1+p^{-1}X^{\frac{1}{2d(G)}}),
\end{align*}
by Lemma \ref{lem:preimage rG}, Lemma \ref{assumption}  and Lemma \ref{DDtilde}.
We note that  the defect of the given pair $(a, b)$ is determined by its reduction modulo $\epsilon$ by Lemma \ref{lem:defect} for $G \in \cG_{\geq 5}$ except $\bZ/2\bZ \times \bZ/6\bZ$ and $\bZ/2\bZ \times \bZ/8\bZ$, and modulo $\epsilon^6$ for $G=\bZ/2\bZ \times \bZ/6\bZ$ and $\bZ/2\bZ \times \bZ /8\bZ$.

We consider the case of $\epsilon > 1$, and for simplicity we assume that $\epsilon$ is prime. 
Let $I_e$ be the set of pairs $(\bZ/\epsilon\bZ)^2$ which has a defect $e$.
Then,
\begin{align*}
	\widetilde{M}_{G, I}^e(X) = \bigsqcup_{I' \in I_e} \widetilde{M}_{G, I, I'}(X),
\end{align*}
where $\widetilde{M}_{G, I, I'}(X)$ is a subset of $\widetilde{M}_{G, I}(X)$ where the additional condition $(a,b)\equiv I'$  $\pmod{\epsilon}$ is imposed.
Since $\epsilon$ is a prime, $e = 1$ or $\epsilon$.
By Lemma \ref{lem:MI for Ggeq} and CRT, 
\begin{align*}
	|\widetilde{M}_{G, I}^{\epsilon}(X)| =  \frac{|I_\epsilon|}{(\epsilon^2-1)}
	\frac{1}{(p^2-1)} \frac{\Area(R(1))}{\zeta(2)}X^{\frac{1}{d(G)}} + O(X^{\frac{1}{2d(G)}} +  p^{-1}X^{\frac{1}{2d(G)}}\log X),
\end{align*}
and  
\begin{align*}
	|\widetilde{M}_{G, I}^1(X)| =\frac{(\epsilon^{2} - 1 - |I_\epsilon|)}{(\epsilon^2-1)}
	\frac{1}{(p^2-1)} \frac{\Area(R(1))}{\zeta(2)}X^{\frac{1}{d(G)}} + O(X^{\frac{1}{2d(G)}} +  p^{-1}X^{\frac{1}{2d(G)}} \log X).
\end{align*}
Therefore,
\begin{align*}
	|\cE_{G, J}(X)| &= \frac{|W_{G,J}|}{p^2-1} 
\frac{	((\epsilon^{\frac{12}{d(G)}} - 1)|I_{\epsilon}| + \epsilon^{2} - 1)}{(\epsilon^2-1)} \frac{1}{r(G)}  \frac{\Area(R(1))}{\zeta(2)} X^{\frac{1}{d(G)}} + O(X^{\frac{1}{2d(G)}} +  p^{-1}X^{\frac{1}{2d(G)}} \log X).
\end{align*}
Similarly, for the groups with no defect, we have
\begin{align*}
	|\cE_{G, J}(X)| &=  \frac{|W_{G,J}|}{p^2-1} \frac{1}{r(G)} \frac{\Area(R(1))}{\zeta(2)}
	X^{\frac{1}{d(G)}} + O(X^{\frac{1}{2d(G)}} +  p^{-1}X^{\frac{1}{2d(G)}}\log X).
\end{align*}
By taking $c(G)=\frac{	((\epsilon^{\frac{12}{d(G)}} - 1)|I_{\epsilon}| + \epsilon^{2} - 1)}{(\epsilon^2-1)} \frac{1}{r(G)}  \frac{\Area(R(1))}{\zeta(2)}$ where the first term exists only if $\epsilon \neq 1$, the claim follows.
When $\epsilon$ is not prime (only appear when $G = \bZ/12\bZ$, $\bZ/2\bZ \times \bZ/6\bZ,$ $\bZ/2\bZ \times \bZ/8\bZ$ by Lemma \ref{lem:defect}), we can compute $c(G)$ similarly.
\end{proof}

Our proof gives $c(G)$ concretely except when $G = \bZ/2\bZ \times \bZ/6\bZ$ and $\bZ/2\bZ \times \bZ/8\bZ$.
Even for such $G$, if one know the defects (see Remark \ref{rem:defect}) then can calculate $c(G)$ precisely.

Proposition \ref{prop:sum of WGI} and Theorem \ref{local_count} analogously give results like Theorem \ref{2tor bad}.
Instead of listing them all, we record the results which will be used in the applications.

\begin{corollary} \label{cor:large G local cond}
For 
$G \in \cG_{\geq 5}$ and a prime $p \geq 5$,
\begin{align} \label{count_torsion}
|\cE_G(X)|&=c(G)X^{\frac{1}{d(G)}}+O(X^{\frac{1}{e(G)}}), \nonumber \\ 
|\cE_{G,p}^a(X)|&= c(G)
\frac{H_G(a,p)}{p^2-1} X^{\frac{1}{d(G)}}+O\left(H_G(a,p)X^\frac{1}{e(G)}+\frac{H_G(a,p)}{p}X^{\frac{1}{e(G)}} \log X \right), \\
|\cE_{G,p}^{\mathrm{mult}}(X)|&=O\left(\frac{1}{p}X^{\frac{1}{d(G)}}+ p X^{\frac{1}{e(G)}} +X^{\frac{1}{e(G)}} \log X \right). \nonumber
\end{align}
\end{corollary}

%

Theorems \ref{2tor bad}  and \ref{local_count} gives some results on the probability for elliptic curves with local condition.
 In particular, for $\cLC = \textrm{mult}$, we observe an interesting phenomenon. 
\begin{corollary} \label{p for mult}
The ratios of $c_{G, \mathrm{mult}}(p)$'s for $G \in \mathcal{G}_{\leq 4}$ and $p \geq 5$ are proportional to the number of cusps of the corresponding modular curves.
Also, there is a set of primes $S$ with positive density such that the ratios of $c_{G, \mathrm{mult}}(p)$'s for $G \in \mathcal{G}_{\geq 5}$ are proportional to the number of cusps of the corresponding modular curves when $p \in S$.
\end{corollary}
\begin{proof}
One can easily compute that the numbers of cusps of modular curve $X_1(N)$ for $N = 1, 2, 3, 4$ and $X(2)$ are $1, 2, 2, 3, 3$, respectively (For example, see \cite[\S 3.9]{DS}). So we have the result for $G$ in $\mathcal{G}_{\leq 4}$ by Proposition \ref{prop:sum of WGI} and Theorem \ref{2tor bad}.
Also, the number of cusps of $X_1(N)$ for $N = 5, 6, 7, 8, 9, 10, 12$ and $X_{\Gamma_1(M) \cap \Gamma(2)}$ for $M = 4, 6, 8$ are $4, 4, 6, 6, 8, 8, 10$ and $4, 6, 10$.
For primes $p$ which satisfy the conditions that make $\sum_{4A^3 + 27B^2 \equiv 0} W_{G, (A, B)}$ largest among the possible values in Proposition \ref{prop:sum of WGI}, the proportion of $c_{G, \textrm{mult}}(p)$ for $G$ in $\mathcal{G}_{\geq 5}$ is coincide with above values.
Now Theorem \ref{local_count} and Chebotarev density theorem give $G$ in $\mathcal{G}_{\geq 5}$ part.
\end{proof}


It is well-known that every elliptic curve with torsion $G \in \mathcal{G}_{\geq 5}$ has semistable reduction at $p \nmid |G|$. We can confirm this phenomenon with probability 1. Also, we have the analogous result for $G$ in $\cG_{\leq 4}$.
\begin{corollary} \label{cor:semi} 
For $G$ in $\mathcal{G}_{\geq 5}$ and a prime $p \nmid 6|G|$, we have
\begin{align*}
\lim_{X \rightarrow \infty} \frac{|\cE_{G,p}^{\mathrm{ss}}(X)|}{|\cE_G(X)|}=1.
\end{align*}
 For a torsion subgroup $G$ in $\mathcal{G}_{\leq 4}$ and a prime $p \geq 5$,
\begin{equation*}
c_{G,\mathrm{ss}}(p)=1- \frac{1}{p^2}.
\end{equation*}
\end{corollary}

As we can see in the \cite[Theorem 1.1]{CJ}, the number of elliptic curves with split and non-split reduction at $p$ are the same for all primes $p$.
When we consider elliptic curves with torsion $G$, this property no more holds.
\begin{corollary} \label{cor: split nonsplit}
For $G = \bZ/3\bZ$ and a prime $p \geq 5$, we have
\begin{equation*}
\lim_{X \to \infty} \frac{|\cE_{\bZ/3\bZ, p}^{\mathrm{split}}(X)|}{|\cE_{\bZ/3\bZ, p}^{\mathrm{mult}}(X)|}= \left\{
\begin{array}{lll}
\frac{1}{2} & \textrm{when } p \equiv 5, 11 \pmod{12}, \\
1 & \textrm{when } p \equiv 1 \pmod{12},  \\
0 & \textrm{when } p \equiv 7 \pmod{12}.
\end{array}
\right.
\end{equation*}
\end{corollary}
We note that Corollaries \ref{cor:semi} and \ref{cor: split nonsplit} follow from Proposition \ref{prop:sum of WGI}, Theorem \ref{2tor bad} and \ref{local_count}.

\bigskip

In Section \ref{sec:app}, we establish the Frobenius Trace formula for elliptic curves when  $G = \bZ/2\bZ$ and  $\bZ/2\bZ \times \bZ/2\bZ$. For this purpose, we need to count elliptic curves with finitely many local conditions. Since its proof is similar with that of \cite[Theorem 8]{CJ}, we just introduce the notations and state the results.

Let $P = \lcrc{p_k}_k$ be a finite set of primes such that $p_k \geq 5$,
and $\cJ = \cJ_P$ be a finite set of 2-tuples $\lcrc{(A_k, B_k)}$ for $A_k, B_k \in \bZ/p_k\bZ$ such that $(A_k, B_k ) \not\equiv (0,0) \pmod{p_k}$.
We define analogously $\cM_{G, \cJ}(X) $, $\cE_{G, \cJ}(X)$,  $\cS_{G, \cJ}(X)$, and so on.
Let 
\begin{equation*}
	W_{G, \cJ} = \prod_{k} W_{G, J_k} \textrm{ for } J_k \equiv (A_k, B_k) \pmod{p_k}.
\end{equation*}
Then,
\begin{proposition} \label{prop:manyprimecE}
For $P = \lcrc{p_k}$ and $\cJ = \lcrc{(A_k, B_k)}$, $2$-tuples of $\bZ/p_k\bZ$ such that $(A_k, B_k) \not\equiv (0,0)$ for all $k$, and $G$ in $\mathcal{G}_{\leq 4}$, we have
\begin{align*}
 |\cE_{G, \cJ}(X)| =  c(G) | W_{G, \cJ} |\prod_k 
 \left( \frac{1}{p_k^{2}}\frac{p_k^{\frac{12}{d(G)}}}{p_k^{\frac{12}{d(G)}}-1} \right)X^{\frac{1}{d(G)}}  
+ O(\prod p_k^{-1}X^{\frac{1}{e(G)}}   + X^{ \frac{1}{12}}) .
 \end{align*}
\end{proposition}

We will denote $\cS = (\cLC_{p_i})$ as a finite set of local conditions $\cLC_{p_i}$. 
When an elliptic curve has the local property corresponding to $\cLC_{p_i}$ at $p_i$ for all local conditions in $\cS$, we say that $E$ satisfies $\cS$. Let
\begin{equation*}
\cE^{\cS}_{G}(X) = \lcrc{(A, B) \in \cE_{G}(X) : E_{A, B} \textrm{ satisfies } \cS},
\end{equation*}
and
\begin{equation*} 
|\cLC_{p}|_G := \lim_{X \to \infty} \frac{|\cE_{G, p}^{\cLC_p}(X)|}{|\cE_{G}(X)|}, \qquad |\cS|_G = \prod_{i}|\cLC_{p_i}|_G.
\end{equation*}

Now we address that the local conditions under the torsion restriction are also independent.

\begin{theorem} \label{manyprime}
Let $P = \lcrc{p_k}$ and $\cS$ be a set of local conditions at $p_k$. Then, we have
\begin{equation*}
|\cE_{G}^{\cS}(X)| =   c(G) |\cS|_G X^{\frac{1}{d(G)}} +
O\left( \left( \prod_k p_k \right) X^{\frac{1}{e(G)}} + \left( \prod_k p_k \right)^2 X^\frac{1}{12} \right) 
\end{equation*}
We replace the exponents $1$ and $2$ of $p_k$ in the error term by $0$ and $1$ respectively when $\cLC$ is $\mathrm{multi}, \mathrm{split},$ or $\mathrm{non}\textrm{-}\mathrm{split}$. When $\cLC$ is $a$ in the Weil bound, $p_k$  and $p_k^2$ are replaced by $H_G(a,p_k)/p_k$ and $H_G(a,p_k)$ respectively. 
\end{theorem}

\section{Proofs of the Main Theorems} \label{sec:app}

\subsection{Boundedness of average analytic rank of elliptic curves with prescribed torsion group}

In this section, we show that average analytic rank of elliptic curves with prescribed torsion $G$ is bounded under the GRH for elliptic curve $L$-functions. Let $\phi$ be an even non-negative function with its Fourier transform $\widehat{\phi}$ compactly supported.  Let $\gamma_E$ denote the imaginary part of a non-trivial zero $\rho_E=\frac 12 +i \gamma_E$ of an elliptic curve $L$-function $L(s,E)$. By the explicit formula, we have
\begin{align*}
	\frac{1}{\left| \cE_{G}(X) \right|}\sum_{E \in \cE_{G}(X)} \sum_{\gamma_E}\phi \left( \gamma_E \frac{\log X }{2 \pi }\right) 
	& =\frac{ \widehat{\phi}(0)}{ \left| \cE_{G}(X) \right|}\sum_{E \in \cE_{G}(X)} \frac{ \log N_E}{\log X} + \frac{2}{\pi}\int_{-\infty}^\infty \phi \left( \frac{\log X \cdot r}{2 \pi } \right) \Re \frac{\Gamma_E'}{\Gamma_E}(\frac 12 +ir)dr \\
	&-\frac{2}{\log X \left| \cE_{G}(X) \right| }\sum_{n=1}^\infty \frac{\Lambda(n)}{\sqrt{n}}\widehat{\phi}\left( \frac{\log n}{\log X}\right) \sum_{E \in \cE_{G}(X)}\hat{a}_E(n) \\
	& \leq \widehat{\phi}(0) -\frac{2}{\log X \left| \cE_{G}(X) \right| }\sum_{n=1}^\infty \frac{\Lambda(n)}{\sqrt{n}}\widehat{\phi}\left( \frac{\log n}{\log X}\right) \sum_{E \in \cE_{G}(X)}\hat{a}_E(n) + O\left( \frac{1}{\log X}\right)\\
	&  \leq \widehat{\phi}(0) - S_1 -S_2 +  O\left( \frac{1}{\log X}\right),
\end{align*}
where 
\begin{equation*}
S_1=\frac{2}{\log X \left| \cE_{G}(X) \right| }\sum_{p} \frac{\log p}{\sqrt{p}}\widehat{\phi}\left( \frac{\log p}{\log X}\right) \sum_{E \in \cE_{G}(X)}\hat{a}_E(p),
\end{equation*}
and
\begin{equation*}
S_2=\frac{2}{\log X \left| \cE_{G}(X) \right| }\sum_{p} \frac{\log p}{p}\widehat{\phi}\left( \frac{2\log p}{\log X}\right) \sum_{E \in \cE_{G}(X)}\hat{a}_E(p^2).
\end{equation*}

From now on, for a positive constant $\sigma$ we specify the test function $\phi$ and $\widehat{\phi}$:
\begin{align*}
\widehat{\phi}(u)=\frac{1}{2}\left(\frac{1}{2}\sigma-\frac{1}{2}|u|\right)   \textrm{ for } |u|\leq \sigma, \quad \textrm{and } \quad \phi(x)=\frac{\sin^2(2\pi \frac 12 \sigma x)}{(2\pi x)^2}.
\end{align*}
Note that $\phi(0)=\frac{\sigma^2}{4}$ and $\widehat{\phi}_n(0)=\frac{\sigma}{4}$.

If we show
\begin{align} \label{sum S1 and S2}
-S_1-S_2= \frac{1}{2}\phi(0)+o(1),
\end{align}
by the positivity of $\phi$, we have
\begin{align} \label{average rank}
\frac{1}{\left| \cE_{G}(X) \right|}\sum_{E \in \cE_{G}(X)} r_E  \leq  \frac 12  + \frac{\widehat{\phi}(0)}{\phi(0)} +o(1)  \leq \frac{1}{2} + \frac{1}{\sigma} +o(1).  
\end{align}

Hence, it is left to show $(\ref{sum S1 and S2})$ holds for each torsion group $G$ with some explicit $\sigma$. For this purpose, we need the following lemmas.

\begin{lemma} \label{S1_torsion} For a torsion group $G$ in $\mathcal{G}_{\geq 5}$,
\begin{align*}
	\sum_{E \in \cE_{G}(X)}\hat{a}_E(p) \ll \left( \frac{X^{\frac{1}{d(G)}}}{p}+ p^2X^{\frac{1}{e(G)}} + pX^{\frac{1}{e(G)}}\log X \right). 
\end{align*}
For $G=\bZ/3\bZ$ or $\bZ/4\bZ$,
\begin{align*}
	\sum_{E \in \cE_{G}(X)}\hat{a}_E(p) \ll
\frac{X^\frac{1}{d(G)}}{p} + pX^\frac{1}{e(G)}+ p^{2}X^{{\frac{1}{12}}}.
\end{align*}
\end{lemma}
\begin{proof}
We know that 
\begin{align*}
	\sum_{E \in \cE_{G}(X)}\hat{a}_E(p)= \sum_{|a| < 2\sqrt{p}} 
	\sum_{\substack{E \in \cE_{G}(X) \\  a_E(p) =a  }}\hat{a}_E(p) 
	+ \sum_{\substack{E \in \cE_{G}(X) \\ \text{$E$ mult at $p$}}}\hat{a}_E(p).
\end{align*}

When $G \in \cG_{\geq 5}$
by Corollary \ref{cor:large G local cond},
\begin{align*}
\left| \sum_{\substack{E \in \cE_{G}(X) \\ \text{$E$ mult red at $p$}}}\hat{a}_E(p) \right| 
&\ll \frac{1}{p^\frac 32}X^{\frac{1}{d(G)}}+ p^\frac 12 X^{\frac{1}{e(G)}} +\frac{1}{p^\frac 12}X^{\frac{1}{e(G)}} \log X.
\end{align*}

By Corollary \ref{cor:large G local cond}, $(\ref{moment_torsion 1})$ and $(\ref{moment_torsion 2})$,
\begin{align*}
\sum_{|a| < 2\sqrt{p}} \sum_{\substack{E \in \cE_{G}(X) \\ a_E(p)=a}}\hat{a}_E(p)
&=\sum_{|a| < 2\sqrt{p}}\frac{a}{\sqrt{p}}\left ( c(G)\frac{H_G(a,p) }{p^2-1}  X^{\frac{1}{d(G)}}+O\left(H_G(a,p)X^\frac{1}{e(G)}+\frac{H_G(a,p)}{p}X^{\frac{1}{e(G)}} \log X  \right)\right) \\
&\ll \frac{X^{\frac{1}{d(G)}}}{p}+ p^2X^{\frac{1}{e(G)}} + pX^{\frac{1}{e(G)}}\log X.
\end{align*}

For $G=\bZ/3\bZ$ or $\bZ/4\bZ$, by Theorem \ref{2tor bad},
\begin{align*}
\left| \sum_{\substack{E \in \cE_{G}(X) \\ \text{$E$ mult red at $p$}}}\hat{a}_E(p) \right| 
& \ll  \frac{1}{\sqrt{p}}\left( \frac{1}{p}X^{\frac{1}{d(G)}}+ X^{\frac{1}{e(G)}} +pX^{\frac{1}{12}} \right) \ll \frac{1}{p^\frac 32}X^{\frac{1}{d(G)}}+ \frac{1}{p^\frac 12} X^{\frac{1}{e(G)}} + p^\frac 12 X^{\frac{1}{12}} .
\end{align*}
By Theorem \ref{2tor bad}, $(\ref{moment_torsion 1})$ and $(\ref{moment_torsion 2})$,
\begin{align*}
\sum_{|a| < 2\sqrt{p}} \sum_{\substack{E \in \cE_{G}(X) \\ a_E(p)=a}}\hat{a}_E(p)&=\sum_{|a| < 2\sqrt{p}}\frac{a}{\sqrt{p}} \left (c(G) \frac{H_G(a,p)}{p^2}\frac{p^{\frac{12}{d(G)}}}{p^{\frac{12}{d(G)}}-1}X^{\frac{1}{d(G)}}+O\left(H_G(a,p)(p^{-1}X^\frac{1}{e(G)}+ X^{\frac{1}{12}} ) \right)\right) \\
&\ll
\frac{X^\frac{1}{d(G)}}{p} + pX^\frac{1}{e(G)}+ p^{2}X^{\frac{1}{12}}.
\end{align*}
\end{proof}

\begin{lemma} \label{S2_torsion} For a torsion group $G$ in $\mathcal{G}_{\geq 5}$,
\begin{align*}
\sum_{E \in \cE_{G}(X)}\hat{a}_E(p^2)=-c(G)X^{\frac{1}{d(G)}}+ O\left(\frac{1}{p^{\frac 12}}X^{\frac{1}{d(G)}} + p^2X^\frac{1}{e(G)}+ pX^{\frac{1}{e(G)}}\log X \right)
\end{align*}
For $G=\bZ/3\bZ$ or $\bZ/4\bZ$,
\begin{align*}
\sum_{E \in \cE_{G}(X)}\hat{a}_E(p^2)=-c(G)X^{\frac{1}{d(G)}}+ O\left(\frac{1}{p^{\frac 12}}X^{\frac{1}{d(G)}} + pX^\frac{1}{e(G)}+ p^2X^{\frac{1}{12}} \right).
\end{align*}
\end{lemma}

\begin{proof}
We know that
\begin{align*}
\sum_{E \in \cE_{G}(X)}\hat{a}_E(p^2)= \sum_{|a| < 2\sqrt{p}} \sum_{\substack{E \in \cE_{G}(X) \\ a_E(p)=a}}\hat{a}_E(p^2) + \sum_{\substack{E \in \cE_{G}(X) \\ \text{$E$ mult  at $p$}}}\frac{1}{p}.
\end{align*}
By Corollary \ref{cor:large G local cond}, 
\begin{align*}
\sum_{\substack{E \in \cE_{G}(X) \\ \text{$E$ mult at $p$}}}\frac{1}{p} \ll \frac{X^{\frac{1}{d(G)}}}{p^2} + X^\frac{1}{e(G)}+\frac{X^{\frac{1}{e(G)}}\log X}{p}
\end{align*}
and
\begin{align*}
&\sum_{|a| <  2\sqrt{p}} \sum_{\substack{E \in \cE_{G}(X) \\ a_E(p)=a}}\hat{a}_E(p^2)=\sum_{|a| <  2\sqrt{p}} \sum_{\substack{E \in \cE_{G}(X) \\ a_E(p)=a}}(\hat{a}_E(p)^2-2)\\
&=\sum_{|a| < 2\sqrt{p}}\left(\frac{a^2}{p}- 2 \right)\left( c(G)\frac{H_G(a,p)}{p^2-1}X^{\frac{1}{d(G)}}+O\left(H_G(a,p)X^\frac{1}{e(G)} + \frac{H_G(a,p)}{p}X^{\frac{1}{e(G)}} \log X \right) \right)\\
&=c(G)\frac{\sum_{|a| < 2\sqrt{p}}a^2H_G(a,p)}{p(p^2-1)}X^{\frac{1}{d(G)}}-2c(G)\frac{\sum_{|a| < 2\sqrt{p}}H_G(a,p)}{(p^2-1)}X^{\frac{1}{d(G)}}+O( p^2X^\frac{1}{e(G)}+ pX^{\frac{1}{e(G)}}\log X)\\
&=-c(G)X^{\frac{1}{d(G)}}+ O\left(\frac{1}{p^{\frac 12}}X^{\frac{1}{d(G)}} + p^2X^\frac{1}{e(G)}+ pX^{\frac{1}{e(G)}}\log X \right),
\end{align*}
by Corollary \ref{cor:large G local cond} and $(\ref{moment_torsion 1})$ and $(\ref{moment_torsion 3})$.

For $G=\bZ/3\bZ$ or $\bZ/4\bZ$, by Theorem \ref{2tor bad}, $(\ref{moment_torsion 1})$ and $(\ref{moment_torsion 3})$, similarly we can show that
\begin{align*}
\sum_{E \in \cE_{G}(X)}\hat{a}_E(p^2)=-c(G)X^{\frac{1}{d(G)}}+ O\left(\frac{1}{p^{\frac 12}}X^{\frac{1}{d(G)}} + pX^\frac{1}{e(G)}+ p^2X^{\frac{1}{12}} \right).
\end{align*}
%
\end{proof}


By Lemma \ref{S1_torsion}, for $G$ in $\mathcal{G}_{\geq 5}$,
\begin{align} \label{S1_sum}
S_1 &\ll  \frac{1}{\log X} \sum_{p}  \frac{\log p}{\sqrt{p}} \widehat{\phi}\left( \frac{\log p}{\log X} \right) \left( \frac{1}{p} +p^2 X^{-\frac{1}{e(G)}} +pX^{-\frac{1}{e(G)} } \log X\right) \\
& \ll X^{-\frac{1}{e(G)} } \sum_{ p \leq X^\sigma} p^\frac 32 \log p \ll X^{-\frac{1}{e(G)} + \frac{5\sigma}{2}} \nonumber
\end{align}
and for $G=\bZ/3\bZ$ or $\bZ/4\bZ$, 
\begin{align} \label{S1_sum 2}
S_1 &\ll  \frac{1}{\log X} \sum_{p}  \frac{\log p}{\sqrt{p}} \widehat{\phi}\left( \frac{\log p}{\log X} \right) \left( \frac{1}{p} + pX^{\frac{1}{e(G)}-\frac{1}{d(G)}}+ p^{2}X^{\frac{1}{12}-\frac{1}{d(G)}}\right) \\
& \ll X^{-\frac{1}{d(G)} } \sum_{ p \leq X^\sigma}\left(  p^\frac 12 \log p X^\frac{1}{e(G)} + p^{\frac 32}\log p X^{\frac{1}{12}} \right) \ll X^{-\frac{1}{d(G)}}\left( X^{\frac{1}{e(G)}+\frac{3\sigma}{2}} + X^{\frac{1}{12}+\frac{5\sigma}{2}} \right). \nonumber
\end{align}

By Lemma \ref{S2_torsion}, for $G$ in $\mathcal{G}_{\geq 5}$,
\begin{align*} 
S_2 &= \frac{2}{\log X}\sum_{p} \frac{\log p}{p}\widehat{\phi}\left( \frac{2\log p}{\log X}\right) \left( -1+ O\left(\frac{1}{p^{\frac 12}} + p^2X^{-\frac{1}{e(G)}}+ pX^{-{\frac{1}{e(G)}}}\log X\right) \right)\\
&=-\frac{2}{\log X}\sum_{p} \frac{\log p}{p}\widehat{\phi}\left( \frac{2\log p}{\log X}\right)+ O\left( \frac{1}{\log X} + \sum_{p \leq X^{\frac{\sigma}{2}} } p \log p X^{-\frac{1}{e(G)}}  \right) \\
&=-\frac{1}{2}\phi(0) + O\left(\frac{1}{\log X} + X^{-\frac{1}{e(G)} + \sigma} \right)
\end{align*}
and for $G=\bZ/3\bZ$ or $\bZ/4\bZ$,
\begin{align*} 
S_2 &= \frac{2}{\log X}\sum_{p} \frac{\log p}{p}\widehat{\phi}\left( \frac{2\log p}{\log X}\right) \left(-1+ O\left(\frac{1}{p^{\frac 12}}+ pX^{\frac{1}{e(G)}-\frac{1}{d(G)}}+ p^2X^{\frac{1}{12}-\frac{1}{d(G)}}\right)\right)\\
&=-\frac{2}{\log X}\sum_{p} \frac{\log p}{p}\widehat{\phi}\left( \frac{2\log p}{\log X}\right)+ O\left(\sum_{p \leq X^{\frac{\sigma}{2}} } \log p X^{\frac{1}{e(G)}-\frac{1}{d(G)}} + p\log p X^{\frac{1}{12} -\frac{1}{d(G)} }  \right) \\
&=-\frac{1}{2}\phi(0) + O\left(  X^{\frac{1}{e(G)}-\frac{1}{d(G)} + \frac{\sigma}{2}} + X^{\frac{1}{12}-\frac{1}{d(G)} + \sigma } \right).
\end{align*}

From our computation, if we take  $\sigma=\frac{1}{18}, \frac{1}{18}$, and $\frac{1}{5d(G)}$  for $G=\bZ/3\bZ$, $\bZ/4\bZ$ and $G$ in $\mathcal{G}_{\geq 5}$ respectively then (\ref{sum S1 and S2}) and  (\ref{average rank}) hold. Therefore, the average of ranks is bounded by
\begin{align*}
18+\frac{1}{2}, \qquad  18+\frac{1}{2}, \qquad \text{and} \qquad 5d(G)+
 \frac 12
 \end{align*} 
for $G=\bZ/3\bZ$, $\bZ/4\bZ$ and $G$ in $\mathcal{G}_{\geq 5}$ respectively and we obtain Theorem \ref{thm1} except for the cases of $G=\bZ/2\bZ$ and $\bZ/2\bZ \times \bZ/2\bZ$, which will be treated in the next section.

\subsection{Trace formula for elliptic curves with torsion points}

In this section we assume that $G=\bZ/2\bZ$  or $\bZ/2\bZ \times \bZ/2\bZ$.  

\begin{theorem}\label{traceF}[Frobenius Trace  Formula for Elliptic Curves] Let $G=\bZ/2\bZ$  or $\bZ/2\bZ \times \bZ/2\bZ$, $k$ be a fixed positive integer. 
Assume $e_i=1$ or $2$, and $r_i$ is odd or $2$ if $e_i=1$, $r_i=1$ if $e_i=2$ for $i=1,\dots ,k$. Then, 
\begin{align*} 
 \sum_{E \in \cE_G(X)} \widehat{a_E}(p_1^{e_1})^{r_1} \widehat{a_E}(p_2^{e_2})^{r_2} \cdots \widehat{a_E}(p_k^{e_k})^{r_k}&=c \frac{c(G)}{\zeta(12/d(G))}X^\frac{1}{d(G)} +O_k\left( \left( \sum_{i=1}^k \frac{1}{p_i} \right) X^{\frac{1}{d(G)} } \right)\\
&+O_k\left( \left( \prod_{i=1}^k p_i \right) X^{\frac{1}{e(G)}} + \left( \prod_{i=1}^k p_i \right)^2 X^{\frac{1}{12}} \right)  
\end{align*}
where
\begin{align*}
c=\left\{ \begin{array}{rl}
 0 & \mbox{ if $e_j=1$ and $r_j$ is odd for some $j$}, \\
 -1 & \mbox{if  $r_j=2$ for all $j$ with $e_j=1$, and the number of $j$'s with $e_j=2$ is odd,} \\
 1 & \mbox{otherwise.}
 \end{array} \right. 
\end{align*}
and the first error term exists only if $e_i=1$ and $r_i=2$ or $e_i=2$ for all $i$. 


\end{theorem}
\begin{proof} First, we consider the case $e_j=1$ and $r_j$ is odd for some $j$. WLOG, we can assume that $e_1=1$ and $r_1$ is odd. We fix local conditions at primes $p_j$, $j=2,3,\dots, k$ and the local condition at $p_1$ is $a_E(p_1)=a$. By Theorem \ref{manyprime}, there are 
\begin{align*}
c(G)  \frac{H_G(a,p_1)}{p_1^2}|S'|_GX^\frac{1}{d(G)}+O\left( \frac{H_G(a,p_1)}{p_1}  (\prod_{i=2}^k c_1(p_i) )X^\frac{1}{e(G)}  + H_G(a,p_1)  (\prod_{i=2}^k c_2(p_i) )X^\frac{1}{12} \right)
\end{align*}
such elliptic curves in $\cE_G(X)$, and $S'$ is the set of the fixed local conditions at $p_i$, $i=2,3,\dots, k$. For the values of  $c_1(p_i)$ and $c_2(p_i)$, we refer to Theorem \ref{manyprime}. Since $\widehat{a_E}(p_2^{e_2})^{r_2} \cdots \widehat{a_E}(p_k^{e_k})^{r_k}$ is a constant and $\sum_{a}a^{r_1}H_G(a,p_1)=0$ for odd $r_1$, only the error term above generates a contribution to the sum. Due to  $\sum_{a}H_G(a,p)=p^2+O_G(p)$, we can see that the total contribution from the error term is at most  $O\left( \left( \prod_{i=1}^k p_i \right) X^\frac{1}{e(G)} + \left( \prod_{i=1}^k p_i \right)^2 X^\frac{1}{12}\right)$.

Next, we need to deal with the case of bad prime $p_1$. Since $a_E(p)=0$ when $E$ has additive reduction at $p$, it is enough to consider the left two local conditions, which is split and non-split. Since the number of elliptic curves with split reduction at $p_1$ and that of elliptic curves with non-split reduction at $p_1$ is the same up to an error term, by the similar argument above, the contribution comes from the error term and is at most
$$
O_k\left( \frac{1}{ p_1^\frac{r_1}{2} } (\prod_{i=2}^k p_i ) X^\frac{1}{e(G)} +\frac{p_1}{ p_1^\frac{r_1}{2} } (\prod_{i=2}^k p_i )^2 X^\frac{1}{12}\right)
$$
and the case $e_j=1$ and $r_j$ is odd for some $j$  is done. 

The next case we treat is $e_i=1$ and $r_i=2$ for all $i$. First, we compute the contribution from good primes by imposing the local conditions $\cLC_{p_i}=a_i$  for all $i=1,\dots,k$ and varying the $a_i$ within the Weil bound $|a_i| < 2\sqrt{p_i}$. The corresponding contribution is by Theorem \ref{manyprime}
\begin{align*}
\left( \prod_{i=1}^k\frac{p_i^{\frac{12}{d(G)}}}{p_i^3 (p_i^{\frac{12}{d(G)}}-1)}\right)\frac{c(G)}{\zeta(12/d(G))}X^{\frac{1}{d(G)}}\cdot \sum_{|a_i| < 2\sqrt{p_i}}a_1^2H_G(a_1, p_1)a_2^2 H_G(a_2,p_2)\cdots a_k^2 H_G(a_k,p_k)\\
+O\left(\sum_{\substack{ |a_i| <  2\sqrt{p_i}}} \left[ \prod_{i=1}^k \frac{a_i^2H_G(a_i,p_i)}{p_i^2} X^{\frac{1}{e(G)}} + \prod_{i=1}^k \frac{a_i^2H_G(a_i,p_i)}{p_i} X^{\frac{1}{12}}\right] \right),
\end{align*}
which is, by the identity $\sum_{|a| < 2\sqrt{p} }a^2 H_G(a,p)=p^3+O_G(p^2)$, 
\begin{align*}
\frac{c(G)}{\zeta(12/d(G))}X^{\frac{1}{d(G)}}+O_k\left( \lbrb{ \sum_{i=1}^k \frac{1}{p_i} } X^{\frac{1}{d(G)}} + \lbrb{\prod_{i=1}^k p_i } X^{\frac{1}{e(G)}} + \lbrb{\prod_{i=1}^k p_i }^2 X^{\frac{1}{12}} \right).
\end{align*}
When $\cLC_{p_i}$ is multi, $\widehat{a_E}(p_i)^2=\frac{1}{p_i}$. Then using the trivial bound $\widehat{a_E}(p_i)^2\leq 4 $ for the other primes $p_j$, the contribution for this case is 
\begin{align*}
 \ll_k \left(\sum_{i=1}^k  \frac{1}{p_i^2} \right)X^\frac{1}{d(G)} +  \left( \sum_{i=1}^k  \frac{1}{p_i}  \right) X^\frac{1}{e(G)} + X^{\frac{1}{12}}. 
\end{align*}

The last case is when $e_1=e_2=\cdots=e_l=2$ and $e_j=1$ and $r_j=2$ for $j>l$ for some $i\leq l \leq k$. Note that $\widehat{a_E}(p^2)=\widehat{a_E}(p)^2-2$ for $E$ with good reduction at $p$ and $\widehat{a_E}(p^2)=\widehat{a_E}(p)^2$ for $E$ with bad reduction at $p$. Hence, it is enough to consider elliptic curves with good reduction at all the primes $p_i$'s.  This amounts to 
\begin{align*}
 \sum_{\text{$E$ has good reduction at $p_i$'s}} (\widehat{a_E}(p_1)^2-2) \cdots (\widehat{a_E}(p_l)^2-2) \widehat{a_E}(p_{l+1})^{2} \cdots \widehat{a_E}(p_k)^{2},
\end{align*} 
which is equal to 
\begin{align*}
(-1)^l \frac{c(G)}{\zeta(12/d(G))} X^\frac{1}{d(G)}+ O_k\left( \left( \sum_{i=1}^k \frac{1}{p_i}\right) X^\frac{1}{d(G)} + \left( \prod_{i=1}^k p_i\right)X^\frac{1}{e(G)} +\left( \prod_{i=1}^k p_i\right)^2X^\frac{1}{12} \right)
\end{align*}
by the result of the previous case and the identity $(1-2)^l=(-1)^l$. 
\end{proof}

\subsection{The distribution of analytic ranks of elliptic curves}

From now on, assume that every elliptic curve $L$-function satisfies Generalized Riemann Hypothesis.  Let $\gamma_E$ denote the imaginary part of a non-trivial zero of $L(s,E)$. We index them using the natural order in real numbers:
\begin{align*}
\cdots \gamma_{E,-3} \leq \gamma_{E,-2} \leq \gamma_{E,-1} \leq \gamma_{E,0} \leq \gamma_{E,1} \leq \gamma_{E,2} \leq \gamma_{E,3} \cdots
\end{align*}
if analytic rank $r_E$ is odd,
\begin{align*}
\cdots \gamma_{E,-3} \leq \gamma_{E,-2} \leq \gamma_{E,-1} \leq 0  \leq \gamma_{E,1} \leq \gamma_{E,2} \leq \gamma_{E,3} \cdots
\end{align*}
otherwise.

In this section, we also assume that $G=\bZ/2\bZ$ or $\bZ/2\bZ \times \bZ/2\bZ$. For elliptic curves in $\cE_G$, we obtain an upper bound on every $n$-th moment of analytic ranks and as a corollary, we show that there are not so many elliptic curves with a high rank.  For this purpose, we compute an $n$-level density with multiplicity. 
The main reference is \cite[Part VI]{Mil}.

For the $n$-level denisty, we choose the same test function for  some $\sigma_n$ in the previous section:
\begin{align*}
\widehat{\phi}_n(u)=\frac{1}{2}\left(\frac{1}{2}\sigma_n-\frac{1}{2}|u|\right)   \textrm{ for } |u|\leq \sigma_n, \quad \textrm{and } \quad \phi_n(x)=\frac{\sin^2(2\pi \frac 12 \sigma_n x)}{(2\pi x)^2}.
\end{align*}
Note that $\phi_n(0)=\frac{\sigma_n^2}{4}$, $\widehat{\phi}_n(0)=\frac{\sigma_n}{4}$ and
\begin{align} \label{estimate}
\int_\mathbb{R}|u|\widehat{\phi}_n(u)^2 du= \frac{1}{6}\phi_n(0)^2.
\end{align}
We show that the $n$-level density holds by taking $\sigma_n= \frac{1}{9n}$ and $\frac{1}{10n}$ for $G=\bZ/2\bZ$ and $\bZ/2\bZ \times \bZ/2\bZ$ respectively.

The $n$-level density with multiplicity is 
\begin{eqnarray*}
D_n^*(\cE_{G}, \Phi)=\frac{1}{|\cE_G(X)|}\sum_{E\in \cE_G(X)} \sum_{j_1,j_2,\dots, j_n}\phi_n \left( \gamma_{E, j_1} \frac{\log X}{2\pi}\right) \phi_n \left( \gamma_{E, j_2} \frac{\log X}{2\pi}\right)\cdots \phi_n \left( \gamma_{E, j_n} \frac{\log X}{2\pi}\right),
\end{eqnarray*}
where $\gamma_{E, j_k}$ is an imaginary part of $j_k$-th zero of $L(s, E)$.
Then, trivially we have
\begin{eqnarray} \label{moment-ineq}
\frac{1}{|\cE_{G}(X)|}\sum_{E\in \cE_{G}(X)} r_E^n \leq \frac{1}{\phi_n(0)^n}D_n^*(\cE_{G}, \Phi).
\end{eqnarray}
By the same argument in \cite[\S 4]{CJ}, we have
\begin{align*}
D_n^*(\cE_{G}, \Phi) 
&\leq \frac{1}{|\cE_{G}(X)|}\sum_{S} \widehat{\phi}_n(0)^{|S^c|} \left( -\frac{2}{\log X} \right)^{|S|}\\ 
& \times   \sum_{m_{i_1}, m_{i_2}, \dots, m_{i_k}} \frac{\Lambda(m_{i_1})\Lambda(m_{i_2}) \cdots \Lambda(m_{i_{k}}) }{ \sqrt{m_{i_1}m_{i_2} \cdots m_{i_{k}}}} \widehat{\phi}_n \left( \frac{\log m_{i_1}}{\log X}\right) \cdots \widehat{\phi}_n \left( \frac{\log m_{i_k}}{\log X}\right)\\
& \times \sum_{ E \in \cE_{G}(X)} \widehat{a}_E(m_{i_1})\widehat{a}_E(m_{i_2}) \cdots \widehat{a}_E(m_{i_k}) + O\left(\frac{1}{\log X} \right),
\end{align*}
where $m_i$'s are primes or squares of a prime with $m_i \leq X^{\sigma_n}$ and $S=\{ i_i, i_2, \dots, i_k \}$ runs over every subset of  $\{1,2,3, \cdots, n\}$. Using the Frobenius trace formula (Theorem \ref{traceF}), we can prove the following propositions as we did in \cite[Proposition 4.1, 4.2]{CJ}.

\begin{proposition} \label{non-square} Let $\widehat{\phi}$ be as above with $\sigma_n=\frac{1}{9n}$ and $\frac{1}{10n}$ for $G=\bZ/2\bZ$ and $\bZ/2\bZ \times \bZ/2\bZ$ respectively. Then, we have
\begin{align*}
&\sum_{E \in \cE_{G}(X)} \sum_{m_{i_1} m_{i_2} \dots m_{i_k} \neq \square} \frac{\Lambda(m_{i_1})\cdots \Lambda(m_{i_{k}}) \widehat{a}_E(m_{i_1})\cdots \widehat{a}_E(m_{i_k})}{\sqrt{m_{i_1}m_{i_2} \cdots m_{i_{k}}}} \widehat{\phi}_n \left( \frac{\log m_{i_1}}{\log X}\right) \cdots \widehat{\phi}_n \left( \frac{\log m_{i_k}}{\log X}\right) \\
&\ll  |\cE_{G}(X)|. 
\end{align*}
\end{proposition}
\begin{proof}
Note that $\widehat{a}_E(m_{i_1})\widehat{a}_E(m_{i_2})\cdots \widehat{a}_E(m_{i_k})$ is of the form
$$
\widehat{a}_E(p_1)^{e_1}\widehat{a}_E(p_2)^{e_2}\cdots \widehat{a}_E(p_t)^{e_t}  \widehat{a}_E(q_1^2)^{l_1}  \widehat{a}_E(q_2^2)^{l_2} \cdots  \widehat{a}_E(q_s^2)^{l_s} , 
$$
with with $e_1+\cdots +e_t+ l_1  +\cdots+ l_s=k$. Here $p_1,p_2, \dots, p_t$ are distinct primes and $q_1,q_2,\dots, q_s$ are distinct primes, but some $q_j$ might be equal to some $p_i$. For a while we assume that the primes $p_1,\dots, p_t, q_1,\dots, q_s$ are all distinct. 

By our assumption, one of $e_i$'s is odd. In this case, the proof of Theorem \ref{traceF} also works and we have,
\begin{align*}
\sum_{ E \in \cE(X)} \widehat{a}_E(p_1)^{e_1}\widehat{a}_E(p_2)^{e_2}\cdots \widehat{a}_E(p_t)^{e_t}  \widehat{a}_E(q_1^2)^{l_1}  \widehat{a}_E(q_2^2)^{l_2} \cdots  \widehat{a}_E(q_s^2)^{l_s} \\
=O( p_1p_2\cdots p_t q_1 q_2 \cdots q_s X^{\frac{1}{e(G)}} + (p_1p_2\cdots p_t q_1 q_2 \cdots q_s)^2 X^{\frac{1}{12}}). 
\end{align*}

The contribution of this case in the worst situation is at most
\begin{align*}
&\ll  X^{\frac{1}{e(G)}}  \left( \sum_{p < X^{\sigma_n}}  p^{\frac{1}{2} } \log p \right)^k  + X^{\frac{1}{12}}  \left( \sum_{p < X^{\sigma_n}}  p^{\frac{3}{2} } \log p \right)^k  \\ 
&\ll  X^{\frac{1}{e(G)}} ( X^{\frac 32 \sigma_n})^n + X^{\frac{1}{12}} ( X^{\frac 52 \sigma_n})^n \ll X^{\frac{1}{d(G)}}.
\end{align*}
where the last inequlity holds by taking $\sigma_n=\min\left( \frac{2}{3n}\left(\frac{1}{d(G)}-\frac{1}{e(G)} \right), \frac{2}{5n}\left(\frac{1}{d(G)}-\frac{1}{12} \right)\right)$, which are $\frac{1}{9n}$ and $\frac{1}{10n}$ respectively.

Now, we assume that some $p_i$ is equal to some $q_j$. Since $\widehat{a}_E(q^2)^{l}=(\widehat{a}_E(q)^2-2)^l$ if $E$ has good reduction at $q$ and  $\widehat{a}_E(q^2)^{l}=\widehat{a}_E(q)^{2l}$ otherwise, still we can use the Frobenius trace formula. 
\end{proof}

\begin{proposition}\label{square} Let $\widehat{\phi}$ be as above with $\sigma_n=\frac{1}{9n}$ and $\frac{1}{10n}$ for $G=\bZ/2\bZ$ and $\bZ/2\bZ \times \bZ/2\bZ$ respectively. For a subset $S=\{ i_1, i_2, \dots, i_k \}$ of $\{1,2,\dots,n\}$,
\begin{align*}
\frac{1}{|\cE_{G}(X)|}\left(\frac{-2}{\log X} \right)^{|S|}\sum_{ E \in \cE_G(X)}\sum_{m_{i_1} m_{i_2} \dots m_{i_k} = \square}\left( \prod_{j=1}^{|S|} 
\frac{\Lambda(m_{i_j})\widehat{a}_E(m_{i_j}) }{\sqrt{m_{i_j}}} \widehat{\phi}_n \left(\frac{\log m_{i_j} }{\log X} \right)\right)\\
=\sum_{\substack{S_2 \subset S  \\ |S_2| \text{even}}} \left( \frac 12 \phi_n(0) \right)^{|S_2^c|} \left| S_2 \right|! \left(\int_\mathbb{R} |u| \widehat{\phi}_n(u)^2du \right)^{\frac{|S_2|}{2} } + O\left( \frac{1}{\log X} \right). 
\end{align*}
\end{proposition}

\begin{proof}
In this proof, we compute the double sum not considering the term $\frac{1}{|\cE(X)|}\left( \frac{-2}{\log X} \right)^{k}$. We show that  every contribution except one is $\ll X^{\frac{1}{d(G)}} (\log X)^{k-1}$, hence they become the error term $O(1/ \log X)$  in the end. 

Note that $\widehat{a}_E(m_{i_1})\widehat{a}_E(m_{i_2})\cdots \widehat{a}_E(m_{i_k})$ is of the form
$$
\widehat{a}_E(p_1)^{e_1}\widehat{a}_E(p_2)^{e_2}\cdots \widehat{a}_E(p_t)^{e_t}  \widehat{a}_E(q_1^2)^{l_1}  \widehat{a}_E(q_2^2)^{l_2} \cdots  \widehat{a}_E(q_s^2)^{l_s} , 
$$
with with $e_1+\cdots +e_t+ l_1 + \cdots + l_s=k$ and $e_i$'s are all even. 
If $e_i \geq 4$ for some $i$ or $l_j \geq 2$ for some $j$, then by the trivial bound, this term is majorized by $X^{\frac{1}{d(G)}}(\log X)^{k-1}$. Let $S_2$ be a subset of $S$ with even cardinality $2t$:
$$
S_2=\{ i_{a_1}, i_{a_2}, \cdots, i_{a_{2t-1}}, i_{a_{2t}} \}, \qquad S_2^c=\{i_{b_1},i_{b_2},\cdots,i_{b_s} \}.
$$
There are $(2t)!/2^t$ ways to pair up two elements in $S_2$. For example, we consider the following pairings.
\begin{align*}
(i_{a_1},i_{a_2}), (i_{a_3},i_{a_4}),  (i_{a_5},i_{a_6}), \cdots, (i_{a_{2t-1}},i_{a_{2t}}).
\end{align*}

This set of pairings corresponds the following sum 
\begin{align*}
 \sum_{ E \in \cE_G(X)} \widehat{a}_E(p_{i_{a_1}})^{2}\widehat{a}_E(p_{i_{a_3}})^{2}\cdots \widehat{a}_E(p_{i_{a_{2t-1}}})^{2}  \widehat{a}_E(q_{i_{b_1}}^2) \widehat{a}_E(q_{i_{b_2}}^2) \cdots  \widehat{a}_E(q_{i_{b_s}}^2)
\end{align*}
where $2t+s=k$. By the Frobenius trace formula, the above sum is
\begin{align*}
|\cE_G(X)|\cdot
\left\{ \begin{array}{clc} 1 &  \mbox{ if $s$ is even,} \\ -1 &  \mbox{ if $s$ is odd} \end{array}\right. + O\left(\left( \frac{1}{p_1}+\cdots+ \frac{1}{p_t} + \frac{1}{q_1} + \cdots \frac{1}{q_s} \right)X^{\frac{1}{d(G)}} \right) \\
+ O\left(p_1\cdots p_t q_1 \cdots q_s X^{\frac{1}{e(G)}}+(p_1\cdots p_t q_1 \cdots q_s)^2 X^{\frac{1}{12}} \right)
\end{align*}

The contribution from the 2nd big O-term  is dominated by
$$
( X^{\sigma_n} \log X )^t ( X^{\frac{\sigma_n}{2}})^{s} X^{\frac{1}{e(G)}} + ( X^{2\sigma_n} \log X )^t ( X^{\sigma_n})^{s} X^{\frac{1}{12}} \ll X^{\frac{1}{d(G)}}(\log X)^t.
$$

The contribution from the error term $O\left(\left( \frac{1}{p_1}+\cdots+ \frac{1}{p_t} + \frac{1}{q_1} + \cdots \frac{1}{q_s} \right)X^{\frac{1}{d(G)}}\right)$ is dominated by $X^{\frac{1}{d(G)}}(\log X)^{k-1}$. The main term of the sum, after being divided by $|\cE_G(X)| \left( \frac{\log X}{-2}\right)^k$,   gives rise to 
\begin{align*} \prod_{i=1}^t \left( \left( \frac{-2}{\log X} \right)^2   \sum_{p} \frac{\log^2 p}{p} \widehat{\phi}_n \left(\frac{\log p}{\log X}\right)^2\right)
\times \prod_{j=1}^s\left(  \frac{2}{\log X} \sum_{q}\frac{\log q}{q} \widehat{\phi}_n \left( \frac{2 \log q }{\log X} \right) \right),
\end{align*}
which equals, by the prime number theorem, 
\begin{align*}
\left( 2^t   \prod_{i=1}^t \int_\mathbb{R} |u| \widehat{\phi}_n(u)^2du\right)\left( \left( \frac{1}{2}\right)^{s} \prod_{j=1}^s \int_\mathbb{R} \widehat{\phi}_n(u)du \right).
\end{align*}
Since there are $(2t)!/2^t$ ways to pair up two elements in $S_2$, the claim follows. 
\end{proof}
By Propositions \ref{non-square} and  \ref{square}, and (\ref{estimate}) we have the following inequality
\begin{align*}
D_n^*(\cE_{G}, \Phi) \leq \phi_n(0)^n \sum_{S}\left( \frac{1}{\sigma_n} \right)^{|S^c|}\sum_{\substack{S_2 \subset S  \\ |S_2| \text{even}}} \left( \frac 12  \right)^{|S_2^c|} \left| S_2 \right|! \left(\frac 16\right)^{\frac{|S_2|}{2}} + O\left( \frac{1}{\log X} \right), 
\end{align*}
and, by $(\ref{moment-ineq})$, we have

\begin{theorem} \label{thm : moment}
Assume GRH for elliptic curve $L$-functions. Let $r_E$ be the analytic rank of an elliptic curve $E$. For every positive integer $n$, we have
\begin{align*}
\limsup_{X \rightarrow \infty} \frac{1}{|\cE_{
G}(X)|}\sum_{E \in \cE_G(X)}r_E^n \leq \sum_{S}\left( \frac{1}{\sigma_n} \right)^{|S^c|}\sum_{\substack{S_2 \subset S  \\ |S_2| \text{even}}} \left( \frac 12  \right)^{|S_2^c|} \left| S_2 \right|! \left(\frac 16\right)^{\frac{|S_2|}{2}}, 
\end{align*}
where $S$ runs over subsets of $\{1,2,3,\dots,n\}$, and $S_2$ runs over subsets of even cardinality of the set $S$. In particular, the average analytic ranks of $\cE_{\bZ/2\bZ}$ is bounded by $9.5$ and the average analytic ranks of $\cE_{\bZ/2\bZ \times \bZ/2\bZ}$ is bounded by $10.5$.
\end{theorem}

Now, we show the sparsity of elliptic curves in $\cE_{G}$ with high analytic ranks. 
We choose the test function $\phi_{2n}(x)$. Then $\widehat{\phi}_{2n}(0)=\frac{1}{4}\sigma_{2n}$, and $\phi_{2n}(0)=\frac 14 \sigma_{2n}^2$.

By Weil's explicit formula, we have
$$r_E \phi_{2n}(0) \leq \widehat{\phi}_{2n}(0) -\frac{2}{\log X} \sum_{m_i}\frac{\widehat{a}_E(m_i)\Lambda(m_i)}{\sqrt{m_i}}\widehat{\phi}_{2n} \left( \frac{\log m_i}{\log X}\right) + O\left( \frac{1}{\log X}\right),$$
hence
\begin{align*}
r_E \leq \frac{1}{\sigma_{2n}} + \frac{4}{\sigma_{2n}^2}\left( -\frac{2}{\log X} \sum_{m_i}\frac{\widehat{a}_E(m_i)\Lambda(m_i)}{\sqrt{m_i}}\widehat{\phi}_{2n} \left( \frac{\log m_i}{\log X}\right) \right) + O\left( \frac{1}{\sigma_{2n}^2 \log X}\right).
\end{align*}

Now assume that $r_E \geq \frac{1+C}{\sigma_{2n}}$ with some positive constant $C$. Then, for sufficiently large $X$, 
\begin{align*}
-\frac{2}{\log X} \sum_{m_i}\frac{\widehat{a}_E(m_i)\Lambda(m_i)}{\sqrt{m_i}}\widehat{\phi}_{2n} \left( \frac{\log m_i}{\log X}\right) \geq \frac{C\sigma_{2n}}{4}.
\end{align*}

Therefore,
\begin{align*}
\left| \{ E \in \cE_{G}(X) | r_E \geq  \frac{1+C}{\sigma_{2n}} \} \right| & \left( \frac{C\sigma_{2n}}{4} \right)^{2n} 
 \leq \sum_{E \in \cE_{G}(X)}  \left( -\frac{2}{\log X} \sum_{m_i}\frac{\widehat{a}_E(m_i)\Lambda(m_i)}{\sqrt{m_i}}\widehat{\phi}_{2n} \left( \frac{\log m_i}{\log X}\right)  \right)^{2n} \\
& \leq \left( \frac{\sigma_{2n}^2}{4}\right)^{2n} \sum_{S_2 \subset \{1,2,3,\dots,2n \}} \left( \frac{1}{2} \right)^{|S_2^c|}|S_2|!\left( \frac 16 \right)^{\frac{|S_2|}{2}}|\cE_{G}(X)| + O\left(\frac{X^{\frac{1}{d(G)}} }{\log X} \right),
\end{align*}
where the second inequality is justified by Propositions \ref{non-square}, \ref{square}, and finally we obtain
\begin{theorem} \label{rank-dist}
Assume GRH for elliptic curve $L$-functions.  Let $C$ be a positive constant, let $n$ a positive integer. We have
\begin{align*}
P\left(r_E \geq  \frac{(1+C)}{\sigma_{2n}} \right)  \leq \frac{\sum_{k=0}^{n}{ {2n} \choose {2k}}\left( \frac 12\right)^{2n-2k}(2k)!\left( \frac 16 \right)^k }{\left( \frac{C}{\sigma_{2n}} \right)^{2n}},
\end{align*}
where $\sigma_{2n}=\frac{1}{18n}$ and $\frac{1}{20n}$ for $G=\bZ/2\bZ$ and $G=\bZ/2\bZ \times \bZ/2\bZ$ respectively.
\end{theorem}

\textbf{Acknowledgement} We would like to thank Dohyeong Kim, John Voight, John Cullinan, and Junyeong Park for the useful discussion, Chan-Ho Kim for introducing the work of Harron and Snowden \cite{HS}, Daniel Fiorilli for his comments.

\section{Appendix} \label{sec:appendix}
Here we summarize $f_G(a, b)$ and $g_G(a, b)$ for all torsion subgroup.
\begin{equation*}
\begin{array}{llll|}
\hline
f_5 &= -27 a^4+324 a^3 b-378 a^2 b^2-324 a b^3-27 b^4,\\
g_5 &= 54 a^6-972 a^5 b+4050 a^4 b^2+4050 a^2 b^4+972 a b^5+54 b^6, \\ \hline
f_6 &= -243 a^4-324 a^3 b-810 a^2 b^2-324 a b^3-27 b^4,\\
g_6 &= -1458 a^6-2916 a^ 5 b+7290 a^4 b^2+9720 a^3 b^3+5346 a^2 b^4+972 a b^5+54 b^6, \\ \hline
f_7 &= -27 a^8+324 a^7 b-1134 a^6 b^2+1512 a^5 b^3-945 a^4 b^4+378 a^2 b^6-108 a b^7-27 b^8 \\
g_7 &= 54 a^{12}-972 a^{11} b+6318 a^{10} b^2-19116 a^9 b^3+30780 a^8 b^4-26244 a^7 b^5+14742 a^6b^6  \\ 
	&\,\, -11988 a^5 b^7+9396 a^4 b^8-2484 a^3 b^9-810 a^2 b^{10}+324 a b^{11}+54 b^{12} \\ \hline
\end{array}
\end{equation*}
\begin{equation*}
\begin{array}{llll|}
\hline
f_8 &= -432 a^8+1728 a^7 b-6048 a^6 b^2+12096 a^5 b^3-12960 a^4 b^4+7776 a^3 b^5-2592 a^2 b^6+432 a b^7-27 b^8 \\
g_8 &= -3456 a^{12}+20736 a^{11} b-190080 a^9 b^3+555984 a^8 b^4-855360 a^7 b^5+840672 a^6 b^6 \\ 
	&\,\, -554688 a^5 b^7+246240 a^4 b^8-71712 a^3 b^9+12960 a^2 b^{10}-1296 a b^{11}+54 b^{12} \\ \hline
f_9 &= -27 a^{12}+324 a^{11} b-1458 a^{10} b^2+3456 a^9 b^3-5103 a^8 b^4+4860 a^7 b^5-3078 a^6 b^6 \\
	&\,\, +972 a^5 b^7+486 a^4 b^8-756 a^3 b^9+324 a^2 b^{10}-27 b^{12} \\
g_9 &= 54 a^{18}-972 a^{17} b+7290 a^{16} b^2-30780 a^{15} b^3+84078 a^{14} b^4-160380 a^{13} b^5+222912 a^{12} b^6 \\
	&\,\, -228420 a^{11} b^7+174960 a^{10} b^8-109728 a^9 b^9+73386 a^8 b^{10}-58320 a^7 b^{11}+39690 a^6 b^{12} \\
	&\,\, -16524 a^5 b^{13}+1458 a^4 b^{14}+2268 a^3 b^{15}-972 a^2 b^{16}+54 b^{18} \\ \hline
\end{array}
\end{equation*}
\begin{equation*}
\begin{array}{llll|}
\hline
f_{10} &= -432 a^{12}+3456 a^{11} b-11232 a^{10} b^2+19440 a^9 b^3-19440 a^8 b^4+7776 a^7 b^5+6912 a^6 b^6\\
	&\,\,\,-11664 a^5 b^7+6480 a^4 b^8-1080 a^3 b^9-432 a^2 b^{10}+216 a b^{11}-27 b^{12} \\
g_{10} &= 3456 a^{18}-41472 a^{17} b+217728 a^{16} b^2-661824 a^{15} b^3+1296000 a^{14} b^4-1767744 a^{13} b^5+1926288 a^{12} b^6 \\
	&\,\,\, -2037312 a^{11} b^7+2133216 a^{10} b^8-1803600 a^9 b^9+981072 a^8 b^{10}-199584 a^7 b^{11}-128304 a^6 b^{12} \\
	&\,\,\, +112752 a^5 b^{13}-32400 a^4 b^14-216 a^3 b^{15}+2592 a^2 b^{16}-648 a b^{17}+54 b^{18} \\ \hline
\end{array}
\end{equation*}
\begin{equation*}
\begin{array}{llll|}
\hline
f_{12}&= -3888 a^{16}+31104 a^{15} b-194400 a^{14} b^2+816480 a^{13} b^3-2269296 a^{12} b^4+4416768 a^{11} b^5-6318000 a^{10} b^6 \\
	&\,\,\, +6855840 a^9 b^7-5747760 a^8 b^8+3753216 a^7 b^9-1907712 a^6 b^{10}+747792 a^5 b^{11}-221616 a^4 b^{12} \\
	&\,\,\, +47952 a^3 b^{13}-7128 a^2 b^{14}+648 a b^{15}-27 b^{16} \\
g_{12}&= -93312 a^{24}+1119744 a^{23} b-2519424 a^{22} b^2-19502208 a^{21} b^3+175146624 a^{20} b^4-738377856 a^{19} b^5 \\
	&\,\,\, +2114216640 a^{18} b^6-4566176064 a^{17} b^7+7806726864 a^{16} b^8-10854518400 a^{15} b^9+12478123872 a^{14} b^{10} \\
	& \,\,\, -11984223456 a^{13} b^{11}  +9676823760 a^{12} b^{12}-6590020032 a^{11} b^{13}+3786612624 a^{10} b^{14} \\
	& \,\,\,-1831706784 a^9 b^{15}+742184208 a^8 b^{16}-249811776 a^7 b^{17}+68988672 a^6 b^{18}-15353712 a^5 b^{19} \\
	& \,\,\, +2682720 a^4 b^{20}-353808 a^3 b^{21}+33048 a^2 b^{22}-1944 a b^{23}+54 b^{24} \\ \hline
	\end{array}
\end{equation*}
\begin{equation*}
\begin{array}{llll|}
\hline
f_{2\times4}&= -27 a^4 - 378 a^2 b^2 - 27 b^4\\
g_{2\times4}& = -54 a^6 + 1782 a^4 b^2 + 1782 a^2 b^4 - 54 b^6 \\ \hline
f_{2\times6}&= -27 a^8 + 1296 a^6 b^2 - 12960 a^4 b^4 - 393984 a^2 b^6 - 62208 b^8
 \\
g_{2\times6} &= 54 a^{12} - 3888 a^{10} b^2 + 85536 a^8 b^4 - 2363904 a^6 b^6 + 43670016 a^4 b^8 + 86593536 a^2 b^{10} - 5971968 b^{12}
 \\ \hline
f_{2\times8} &= -452984832 a^{16}-1811939328 a^{15} b-3170893824 a^{14} b^2-3170893824 a^{13} b^3  -1953497088 a^{12} b^4 \\
	&\,\,\,-707788800 a^{11} b^5-88473600 a^{10} b^6+51314688 a^9 b^7+31961088 a^8 b^8+6414336 a^7 b^9-1382400 a^6 b^{10} \\
	&\,\,\,-1382400 a^5 b^{11}-476928 a^4 b^{12} -96768 a^3 b^{13}-12096 a^2 b^{14} -864 a b^{15}-27 b^{16} \\
g_{2\times8} &= 3710851743744 a^{24}+22265110462464 a^{23} b+61229053771776 a^{22} b^2+102048422952960 a^{21} b^3 \\
	&\,\,\, +114456583471104 a^{20} b^4  +90104118902784 a^{19} b^5+49618146557952 a^{18} b^6+17546820452352 a^{17} b^7  \\
	&\,\,\, +2194711511040 a^{16} b^8-1694163271680 a^{15} b^9-1411953721344 a^{14} b^{10} \\
	&\,\,\, -656375021568 a^{13} b^{11}-246536994816 a^{12} b^{12}-82046877696 a^{11} b^{13} \\
	&\,\,\, -22061776896 a^{10}b^{14}-3308912640 a^9 b^{15} +535818240 a^8 b^{16}+535486464 a^7 b^{17} \\
	&\,\,\,+189278208 a^6 b^{18}+42964992 a^5 b^{19}+6822144 a^4 b^{20}+760320 a^3 b^{21}+57024 a^2 b^{22}+2592 a b^{23}+54 b^{24} \\ \hline
\end{array}
\end{equation*}

\end{document}